\documentclass[10pt,psamsfonts]{amsart}
\usepackage{amsmath}
\usepackage{amsthm}
\usepackage{amssymb}
\usepackage{multicol}
\usepackage{amscd}
\usepackage{amsfonts}
\usepackage{amsbsy}
\usepackage{epsfig,afterpage}
\usepackage[dvips]{psfrag}
\usepackage{subfigure}
\usepackage{epsfig}
\usepackage{amsmath}
\usepackage[margin=3.6cm]{geometry}
\usepackage{srcltx}
\usepackage{amsfonts}
\usepackage{amssymb}
\usepackage{psfrag}
\usepackage{multicol}
\usepackage{verbatim}
\usepackage{graphicx}
\usepackage{amsmath}
\usepackage{amsthm}
\usepackage{amssymb}
\usepackage{amscd}
\usepackage{amsfonts}
\usepackage{amsbsy}
\usepackage{epsfig}
\usepackage[dvips]{psfrag}
\usepackage{multirow}
\usepackage{graphics}
\usepackage{epstopdf}
\usepackage{overpic}
\usepackage{float}
\usepackage[utf8]{inputenc}

\usepackage{graphicx}

\usepackage[colorlinks=true, linkcolor=blue, citecolor=red]{hyperref}

\newtheorem {theorem} {Theorem}
\newtheorem {proposition} [theorem]{Proposition}
\newtheorem {corollary} [theorem]{Corollary}
\newtheorem {lemma}  [theorem]{Lemma}

\newtheorem {definition} [theorem]{Definition}

\theoremstyle{definition}
\newtheorem {example} [theorem]{Example}
\newtheorem {remark} [theorem]{Remark}

\newtheorem{mtheorem}{Theorem}

\usepackage{float}
\usepackage{tikz, wrapfig}
\tikzset{node distance=3cm, auto}
\allowdisplaybreaks

\begin{document}

\title[Limit cycles in PWS with circular switching manifold]
{Limit cycles in piecewise smooth systems with circular switching manifold}

    \address{$^{1}$ Departamento de Matemática, Instituto de Ciências Exatas (ICEx), Universidade Federal de Minas Gerais (UFMG), Av. Pres. Antônio Carlos, 6627, Cidade Universitária - Pampulha, 31270-901, Belo Horizonte, MG, Brazil}
\address{$^{2}$S\~{a}o Paulo State University (Unesp), Institute of Biosciences, Humanities and
	Exact Sciences. Rua C. Colombo, 2265, CEP 15054--000. S. J. Rio Preto, S\~ao Paulo,
	Brazil.}
    \address{$^{3}$Departament de Matemàtiques, Edifici Cc, Universitat Autònoma de Barcelona, 08193 Bellaterra, Barcelona, Spain}
    \author[Gabriel Rondón, Paulo R. da Silva and Jaume Llibre]{Gabriel Rondón$^{1}$, Paulo R. da Silva$^{2}$ and Jaume Llibre$^{3}$}

\email{grondon@ufmg.br}
\email{paulo.r.silva@unesp.br}
\email{jaume.llibre@uab.cat}

\thanks{ .}

\makeatletter
\@namedef{subjclassname@2020}{\textup{2020} Mathematics Subject Classification}
\makeatother


\date{}
\maketitle

\begin{abstract}
We study limit cycles in piecewise complex systems with switching manifold $\mathbb{S}^1$. Using Möbius transformations we establish an equivalence between circular and straight-line discontinuities that preserves periods, stability, and algebraic structure. For piecewise polynomial holomorphic systems we obtain lower bounds on the number of limit cycles via second-order averaging and, for low degrees, via Lyapunov quantities. For piecewise antiholomorphic systems we prove upper bounds: at most $3$ limit cycles in the linear case and $10$ in the quadratic case. We also prove a rigidity theorem: when both components admit classical holomorphic normal forms at the origin no crossing limit cycles exist. Finally, we construct explicit algebraic limit cycles in the circular context, providing, as far as we know the first such examples in the literature.
\end{abstract}


\section{Introduction}

There are several important aspects to understand the dynamics of differential systems such as the existence and bifurcation of limit cycles. In fact, the famous Hilbert's 16th problem is one of the main open problems in the qualitative theory of planar polynomial vector fields, and finding an upper bound for the maximum number of limit cycles in terms of the degree is still a very difficult question. 

In recent years increasing attention has been devoted to the study of limit cycles in piecewise smooth systems (see the monographs \cite{Acary,Brogliato,Kunze}). A subfamily that has attracted particular interest is given by \emph{piecewise holomorphic systems}, because holomorphic functions appear naturally in several areas of applied science, for example in fluid dynamics \cite{BatGK,Mars,Conw}. Dynamically holomorphic systems of the form $\dot{z}=f(z)$ are integrable and do not exhibit limit cycles, which makes them somewhat poor from the qualitative viewpoint. However this changes dramatically in the nonsmooth setting: in \cite{Rondon2022440} the authors showed that piecewise holomorphic
systems can exhibit limit cycles, while in \cite{GASRONSIL1} Gasull et al. obtained limit cycles by analyzing the zeros of the averaging functions and through a degenerate Andronov–Hopf type bifurcation at a monodromic equilibrium point. More recently in \cite{Gasull2025} the authors studied the simultaneous bifurcation of limit cycles for piecewise holomorphic systems separated by a straight line, obtaining Abelian-type integrals that control the bifurcations of periodic orbits after holomorphic and polynomial perturbations.

Most of the existing works assume that the switching manifold is a straight line. In the present paper we study the case where the switching set is the unit circle $\mathbb{S}^1$, a compact one-dimensional real analytic manifold. This appears naturally as a nonlinear generalization of the linear switching manifolds and allows for richer dynamics. Since Möbius transformations are biholomorphic diffeomorphisms of the Riemann sphere $\widehat{\mathbb{C}}=\mathbb{C}\cup\{\infty\}$ mapping circles and lines into circles or lines (see \cite{Beardon,Conw}), we can reduce the general situation to simpler canonical cases. In particular, any circle $\{z\in\mathbb{C}\mid|z-p|=r\}$ can be mapped either onto the unit circle $\mathbb{S}^1=\{z:|z|=1\}$ or onto a straight line, which allows to employ well-established tools from the theory of nonsmooth systems with linear switching manifolds, such as averaging theory, Lyapunov coefficients, and Abelian-type integrals, while at the same time keeping the geometry of the circle in mind for the interpretation of the results.

A Möbius transformation is a biholomorphic diffeomorphism of the Riemann sphere of the form
\begin{equation}\label{mobius_map}
    \phi(z) = \frac{az + b}{cz + d}
\end{equation} 
with $a,b,c,d \in \mathbb{C}$. These transformations form a group under composition. Moreover every Möbius transformation can be expressed as a composition of translations ($\phi(z)=z+\beta$), rotations ($\phi(z)=e^{i\theta}z$), homotheties or dilations ($\phi(z)=\lambda z$ with $\lambda>0$), and the inversion $\phi(z)=1/z$.

Since Möbius transformations are diffeomorphisms, they preserve the qualitative structure of the dynamics under conjugation. 
This makes them particularly suitable for reducing general circular switching manifolds to canonical forms while maintaining the dynamics of interest (see Figure \ref{fig:mobius}).

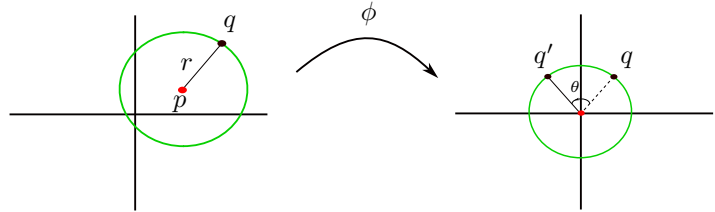
\begin{figure}[h]
\begin{overpic}[scale=0.5]{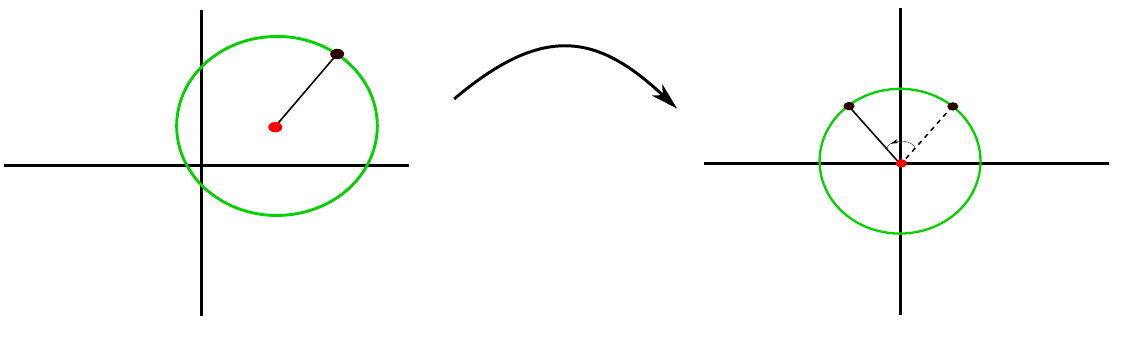}
        \put(30,28){$q$}
        \put(23,17){$p$}
        \put(24,22){$r$}
        \put(85,23){$q$}
        \put(73,23){$q'$}
        \put(49,29){$\phi$}
        \put(78,19){{\tiny$\theta$}}
\end{overpic}
\caption{\footnotesize{We illustrate the effect of Möbius transformations in reducing the general setting to a canonical one. A circle of radius $r$ centered at a point $p \in \mathbb{C}$ is mapped to the unit circle centered at the origin. Furthermore, by composing with a rotation of angle $\theta$, a chosen point $q$ on the boundary can be sent to a new position $q'$ on $\mathbb{S}^1$. Since Möbius maps are diffeomorphisms this reduction preserves the dynamics. Such normalizations allow to simplify the analysis and study the existence of limit cycles bifurcating from points on the unit circle.}}
\label{fig:mobius}
\end{figure}

In this work we will make extensive use of a specific Möbius transformation that maps the unit circle $\mathbb{S}^1$ onto the real line, namely
\begin{equation}\label{mobius_map1}
    \phi(z)=\frac{-z+i}{iz-1}.
\end{equation}
This transformation sends the point $z=i$ to $0$ and, as will be shown it conjugates the piecewise holomorphic system to an equivalent system with switching manifold $\mathbb{R}$ where the theory is well developed.

Motivated by this reduction we now define the piecewise complex system (PWCS) in its canonical form, separated by the unit circle
\begin{equation}\label{ch4:eq111}
\begin{aligned}
\left\{\begin{array}{l}
\dot{z}=F^{+}(z), \quad |z|\geq 1,\\[2pt]
\dot{z}=F^{-}(z), \quad |z|\leq1,
\end{array} \right.
\end{aligned}
\end{equation}
with $z=x+iy$ and $F^{\pm}$ complex-valued functions. Here we concentrate on two fundamental classes of such systems. The first class corresponds to \emph{holomorphic vector fields}, i.e. $F^{+}$ and $F^{-}$ are holomorphic functions of the complex variable $z$. The second class consists of \emph{antiholomorphic vector fields}, i.e. each $F^{\pm}$ is the complex conjugate of a holomorphic function.

Throughout the paper the discontinuity along the unit circle $\mathbb{S}^1$ is understood in the sense of Filippov, see \cite{Filippov88}. Notice that the circle $\Sigma=\mathbb{S}^1$ divides the plane into two regions 
$\Sigma^\pm$ given by $\{z\in\mathbb{C}\mid|z|<1\}$ and $\{z\in\mathbb{C}\mid|z|>1\}$, respectively.

Here we focus on a specific subclass of the holomorphic case: \emph{piecewise polynomial holomorphic systems} (PWHS). These systems are particularly appealing because, on one hand, each smooth component inherits the elegant properties of holomorphic dynamics; on the other hand, the piecewise definition gives rise to phenomena typical of piecewise linear systems, such as sliding motion or limit cycles, and even more complex behavior. The tools of complex analysis are essential in determining their phase portraits, as shown in recent works \cite{GGJ,GGJ2,AMR2170413,Rondon2022440}. In particular, \cite{Rondon2022440} exploits the integrability of holomorphic functions to construct limit cycles in the piecewise setting, while \cite{AMR2170413} uses the Cauchy integral formula to evaluate the Abelian integrals arising in the analysis.
	
While the techniques employed in this work are not exclusive to piecewise holomorphic systems, it is precisely the holomorphic hypothesis that renders the analysis both feasible and elegant. The key observation is that holomorphic vector fields possess a distinctive combination of properties: they are automatically integrable, reversible, and devoid of limit cycles in the smooth regime. Furthermore their singularities admit complete normal forms and all their centers are isochronous. These features make them an ideal testing ground for understanding the new phenomena that appear when a discontinuity is introduced.

The advantage of the holomorphic assumption is also quantitative. A generic polynomial vector field of degree $n$ depends on $n^2+3n+2$ real parameters and may have up to $n^2$ equilibrium points. In contrast a holomorphic polynomial system of degree $n$ is determined by only $2n+2$ real parameters and, unless it is identically zero, has exactly $n$ equilibrium points (counted with multiplicity). This dramatic reduction in complexity is not merely aesthetic: in the piecewise case it leads to strikingly simple expressions for the Lyapunov quantities, which would otherwise be intractable in the general setting.

Let $\mathcal{L}_{n^+,n^-}\in\mathbb{N}\cup\{\infty\}$ denote the maximum number of limit cycles that system \eqref{ch4:eq111} can exhibit in the holomorphic setting. Our main objective is to investigate the number and the nature of limit cycles that can appear in piecewise holomorphic and antiholomorphic systems with circular switching manifold. Note that, by symmetry, $\mathcal{L}_{n^+,n^-}=\mathcal{L}_{n^-,n^+}$ holds trivially.

\subsection{Main Results}

We state our main results. 
\subsubsection{Lower bound for the maximum number of hyperbolic limit cycles bifurcating from a center on $\mathbb{S}^1$}

We focus on the number of limit cycles that can bifurcate from a center lying exactly on the switching manifold $\mathbb{S}^1$. This scenario is depicted in Figure \ref{fig:equilibria_s1} (left): a holomorphic system $\dot{z}=f(z)$ possesses a center point $p$ located on $\mathbb{S}^1$. When perturbed within the class of piecewise holomorphic vector fields sharing $\mathbb{S}^1$ as the switching manifold, the continuum of periodic orbits surrounding $p$ may break, giving rise to limit cycles.

\begin{figure}[h]
\centering
\begin{minipage}{0.32\textwidth}
    \centering
    \begin{overpic}[scale=0.65]{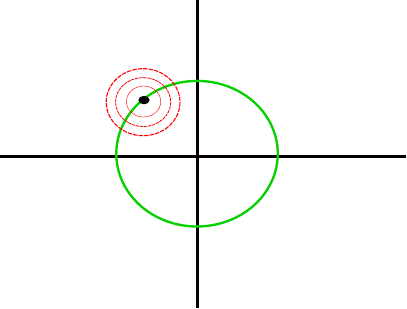}
    \end{overpic}
\end{minipage}
\hfill
\begin{minipage}{0.32\textwidth}
    \centering
    \begin{overpic}[scale=0.65]{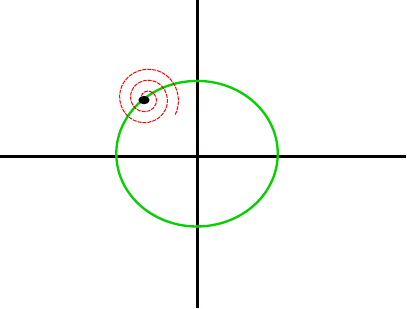}
    \end{overpic}
\end{minipage}
\hfill
\begin{minipage}{0.32\textwidth}
    \centering
    \begin{overpic}[scale=0.65]{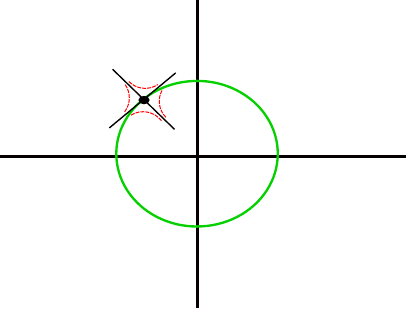}
    \end{overpic}
\end{minipage}
\caption{\footnotesize{Three possible types of equilibrium points lying on the switching manifold $\mathbb{S}^1$: a center (left), a focus (center), and a saddle (right). The center gives rise to a continuum of periodic orbits that can bifurcate into limit cycles under perturbation (Theorem \ref{teo1}). The focus allows for degenerate Hopf bifurcations (Theorem \ref{teonew}). The saddle is the typical configuration for piecewise antiholomorphic systems (Theorem \ref{teo11}).}}
\label{fig:equilibria_s1}
\end{figure}

To construct explicit perturbations that yield lower bounds we work with a concrete unperturbed system that has several desirable properties. We choose $\dot{z}=\frac{1}{2}(1+z^2)$, which has two centers: one at $z=i$ and another at $z=-i$, both lying on $\mathbb{S}^1$. This system is particularly suitable because it is conjugate via a Möbius transformation to a linear center, allowing to apply averaging techniques developed for systems with straight-line discontinuities.

\begin{mtheorem}\label{teo1}
Suppose that $h^\pm$ are holomorphic polynomial functions such that \(\operatorname{deg}(h^\pm) = n^\pm\). Consider the system \eqref{ch4:eq111} with \(F^\pm(z) = \frac{1}{2}(1+z^2) + \epsilon h^\pm\left(\phi(z)\right)\) and the Möbius transformation $\phi$ \eqref{mobius_map1}. Then using averaging theory up to second order, for sufficiently small \(\epsilon\) the maximum number of limit cycles that can bifurcate from the periodic orbits surrounding the center $i$ is given by \(\left\lfloor \frac{3 \max\{n^+, n^-\} + 5}{2} \right\rfloor\)\footnote{\(\lfloor x \rfloor\) is the unique integer satisfying \(\lfloor x \rfloor \le x < \lfloor x \rfloor + 1\).}. This upper bound is achievable, and in this case all these limit cycles are hyperbolic. Consequently we have \(\mathcal{L}_{n^+, n^-} \geq \left\lfloor \frac{3 \max\{n^+, n^-\} + 5}{2} \right\rfloor\).
\end{mtheorem}

Theorem \ref{teo1} provides a lower bound for the maximum number $\mathcal{L}_{n^+, n^-}$ of limit cycles that piecewise holomorphic systems of degrees $n^+$ and $n^-$ can exhibit. This bound comes from analyzing the second-order averaged function, which captures how the perturbation breaks the continuum of periodic orbits into limit cycles. By symmetry, an identical number of limit cycles bifurcates from the center at $z=-i$, though these may coincide with those from $z=i$ depending on the specific perturbation chosen.

\subsubsection{Lower bound for the maximum number of hyperbolic limit cycles bifurcating from foci on $\mathbb{S}^1$}

Our second result concerns the bifurcation of limit cycles from a focus equilibrium point lying on the switching manifold $\mathbb{S}^1$. This situation is illustrated in Figure \ref{fig:equilibria_s1} (center). By composing with a rotation we may assume without loss of generality that this equilibrium is located at $z=i$. Indeed, if the equilibrium is at $p=e^{i\theta_0}\in\mathbb{S}^1$, the rotation $R(z)=e^{i(\pi/2-\theta_0)}z$ satisfies $R(p)=i$, preserves the unit circle, and is holomorphic hence it conjugates the dynamics. This normalization is convenient because the Möbius transformation $\phi$ given in \eqref{mobius_map1} sends $i$ to $0$ and maps $\mathbb{S}^1$ onto the real line, thereby transferring the problem to the well-studied straight-line setting.

Consider the one-parameter unfolding
\[
\dot{z}=F_s^\pm(z),\qquad |z|\gtrless 1,
\]
where $F_s^\pm(z)$ are holomorphic polynomials of degrees $n^\pm$ and $s\in\mathbb{R}^k$ is a vector of small parameters. Assume that for $s=0$ the system has an equilibrium at $z=i$ which is a weak focus of order $k$ for both the inner and outer vector fields, such configuration is called a \emph{weak focus–weak focus} on $\mathbb{S}^1$. 

Taking the pushforward of the unfolding by the Möbius transformation $\phi$ given in \eqref{mobius_map1} we obtain a system on the real line with a weak focus–weak focus at the origin. For this transformed system, the explicit expressions of the first five Lyapunov quantities $V_1,\dots,V_5$ are those computed in \cite{GASRONSIL1}. By virtue of the Möbius equivalence established in Theorem~\ref{thm:mobius_limit_cycle} (which preserves limit cycles and their hyperbolicity), the bifurcation of limit cycles in the original system on $\mathbb{S}^1$ is equivalent to that in the transformed system on $\mathbb{R}$. Therefore, applying the criterion from Proposition~\ref{aux_lemma} (with the known Lyapunov quantities) yields lower bounds for the original system.
\begin{mtheorem}\label{teonew}
Let $\mathcal{L}^0_{n^+,n^-}\le \mathcal{L}_{n^+,n^-}$ denote the maximal number of hyperbolic limit cycles that can bifurcate from a weak focus–weak focus equilibrium on $\mathbb{S}^1$ in a piecewise holomorphic system \eqref{ch4:eq111}. For $n^\pm\in\{1,2\}$ with $n^++n^-\le 4$, the following lower bounds hold:
\[
\mathcal{L}^0_{n^+,n^-}\ge n^++n^-.
\]
In particular, $\mathcal{L}^0_{1,1}\ge 2$, $\mathcal{L}^0_{1,2}=\mathcal{L}^0_{2,1}\ge 3$, and $\mathcal{L}^0_{2,2}\ge 4$.
\end{mtheorem}

\subsubsection{Maximum number of limit cycles of a piecewise anti-holomorphic system}

We now consider piecewise antiholomorphic systems of the form \eqref{ch4:eq111} with $F^\pm(z)=\overline{f^\pm(z)}$, where $f^\pm$ are holomorphic polynomials. These systems possess a Hamiltonian structure which imposes strong constraints on their dynamics. In particular, as illustrated in Figure \ref{fig:equilibria_s1} (right), any nondegenerate equilibrium point of such a system that lies on the switching manifold $\mathbb{S}^1$ is necessarily a saddle. More generally, even when the equilibrium is not located on $\mathbb{S}^1$, the Hamiltonian nature of the system forces the phase portrait to exhibit a symmetric structure that limits the number of possible limit cycles.

By analyzing the half-return maps via resultants we obtain the following upper bounds on the number of limit cycles for low-degree antiholomorphic systems, regardless of whether the equilibrium lies on the switching manifold or not.

\begin{mtheorem}\label{teo11}
    Consider a piecewise antiholomorphic system \eqref{ch4:eq111} with $F^\pm(z)=\overline{f^\pm(z)}$, where $f^\pm$ are holomorphic polynomials. The following statements hold:
    \begin{itemize}
        \item[(a)] If $f^\pm$ are linear, then system \eqref{ch4:eq111} has at most $3$ limit cycles.
        \item[(b)] If $f^\pm$ are quadratic, then system \eqref{ch4:eq111} has at most $10$ limit cycles.
    \end{itemize}
\end{mtheorem}

Recall that Theorem \ref{teo11} applies to general piecewise antiholomorphic systems, without requiring that an equilibrium point lies on the switching manifold $\mathbb{S}^1$. Figure \ref{fig:equilibria_s1} (right) serves to illustrate the typical local behavior near an equilibrium on $\mathbb{S}^1$ when it exists, but the upper bounds hold in all cases. The proof relies on the Hamiltonian structure and resultant computations, which are independent of the location of equilibria relative to the discontinuity.

\subsubsection{PWHS in which each component is in normal form and the switching curve is $\mathbb{S}^1$}
One of the main results of this work is a rigidity theorem concerning the
existence of crossing limit cycles for system~\eqref{ch4:eq111}. 
We consider the case in which the holomorphic singularity is located at the
origin.

Using the holomorphic classification of singularities at the origin we show
that whenever the vector fields on each side of the switching circle admit one
of the classical normal forms, the corresponding piecewise system cannot
generate crossing limit cycles.

This establishes a direct link between the local analytic structure of the
holomorphic singularities at the origin and the global dynamics of the
piecewise system.

   \begin{mtheorem}\label{teod}
Consider the piecewise complex system \eqref{ch4:eq111} where $F^+$ and $F^-$ 
are holomorphic functions defined in punctured neighborhoods of the origin. 
Suppose that each $F^\pm$ satisfies one of the following conditions:

\begin{itemize}
  \item $F^\pm(0) \neq 0$ (regular case);
  \item $F^\pm$ has a simple zero at the origin;
  \item $F^\pm$ has a zero of order $n \ge 2$ at the origin;
  \item $F^\pm$ has a pole at the origin.
\end{itemize}

If $F^\pm$ has a zero of order $n \ge 2$ with nonzero residue, assume 
additionally that $\operatorname{Res}(1/F^\pm, 0) \in \mathbb{R}$ and 
$|\operatorname{Res}(1/F^\pm, 0)| \ge 1$.

Then the system admits no crossing limit cycles.
\end{mtheorem}

\subsubsection{Algebraic limit cycles for PWHS separated by $\mathbb{S}^1$}

A classical notion in the qualitative theory of polynomial differential systems is that of an \emph{algebraic limit cycle}. Informally a limit cycle is said to be algebraic if it is contained (except possibly for a finite set of points) in the zero set of a polynomial. More precisely, for a smooth planar polynomial vector field, a limit cycle \(\Gamma\) is algebraic if there exists an irreducible polynomial \(p(x,y)\) such that \(\Gamma \subset \{(x,y) \in \mathbb{R}^2 \mid p(x,y)=0\}\). The study of such objects has a long history, particularly for quadratic systems, and is intimately related to Darboux integrability theory.

In the piecewise setting the definition must accommodate the fact that the limit cycle may be composed of arcs coming from different vector fields on each side of the switching manifold. Following the ideas introduced in \cite{doi:10.1142/S0218127418500396}, we say that an isolated periodic orbit \(\Gamma\) of a piecewise holomorphic system \eqref{ch4:eq111} with switching manifold \(\Sigma = \mathbb{S}^1\) is an \emph{algebraic limit cycle} if there exist irreducible polynomials \(p^-\) and \(p^+\) (which may coincide) such that
\[
\Gamma \cap (\Sigma^+ \setminus \Sigma) \subset \{ (x,y) \in \mathbb{R}^2 \mid p^+(x,y)=0 \}
\]
and
\[
\Gamma \cap (\Sigma^- \setminus \Sigma) \subset \{ (x,y) \in \mathbb{R}^2 \mid p^-(x,y)=0 \},
\]
 If \(\deg(p^-)=m\) and \(\deg(p^+)=n\), we say that \(\Gamma\) has \emph{bidegree} \((m,n)\) and write \(\Gamma = \{p^-, p^+\}\) for brevity.

In the piecewise linear context fundamental contributions to the study of algebraic limit cycles have been made by Buzzi, Gasull and Torregrosa \cite{doi:10.1142/S0218127418500396}. Among their results they proved that piecewise linear differential systems can exhibit explicit hyperbolic algebraic limit cycles, including examples with two such cycles in the same system. Moreover they showed that for any integers \(m, n \ge 2\) there exist piecewise linear systems possessing a hyperbolic algebraic limit cycle of bidegree \((m, n)\); that is algebraic limit cycles of arbitrarily high degree can be realized.

In this work we extend this concept to piecewise holomorphic systems with switching manifold \(\mathbb{S}^1\) and investigate how algebraic limit cycles behave under Möbius transformations. A fundamental observation is that the integrability properties of holomorphic vector fields, combined with the flexibility of the piecewise definition, allow to construct explicit examples of algebraic limit cycles in the circular setting. This is achieved by first constructing a limit cycle that is an algebraic curve (a circle) in a straight-line switching system and then transferring it via the Möbius transformation \eqref{mobius_map1} to a PWHS with \(\mathbb{S}^1\) as the discontinuity. The following result establishes that such algebraic limit cycles do indeed exist.

\begin{mtheorem}\label{thm:alg_s1}
There exist piecewise holomorphic systems separated by \(\mathbb{S}^1\) that possess algebraic limit cycles.
\end{mtheorem}

The proof is constructive and appears in Section~\ref{sec:alg_cycles}. It relies on the invariance of algebraic curves under Möbius transformations (Proposition~\ref{prop:alg_mobius}) and on an explicit example in the straight-line setting inspired by the techniques developed in \cite{doi:10.1142/S0218127418500396}. This result highlights a key difference with smooth holomorphic systems, which cannot exhibit limit cycles at all, and shows that the presence of the switching manifold \(\mathbb{S}^1\) enables the emergence of this type of structured periodic behavior. In light of the results of Buzzi, Gasull and Torregrosa it is natural to ask whether algebraic limit cycles of arbitrarily high bidegree can also be realized in the PWHS setting, a question that we leave for future investigation.

The paper is organized as follows. Section~\ref{sec:Preliminaries} recalls the necessary background on averaging theory, Lyapunov quantities, local integrability of holomorphic systems, and other auxiliary results. Section~\ref{sec:Mobius} studies Möbius transformations and establishes the equivalence between circular and straight-line switching manifolds. Finally, section~\ref{sec:proofs} contains the proofs of our main results: Theorems~\ref{teo1}, \ref{teonew}, \ref{teo11}, \ref{teod}, and \ref{thm:alg_s1}.

\section{Preliminaries}\label{sec:Preliminaries}
This section presents basic results organized into four subsections.

\subsection{The Averaging Method}\label{aver_metd}

We summarize key aspects of averaging theory for piecewise smooth systems in polar coordinates. Consider systems of the form
\begin{equation}\label{polar_sys}
\frac{dr}{d\theta} = 
\begin{cases}
F^+(\theta,r,\epsilon) & \text{if } 0 \leq \theta \leq \pi, \\
F^-(\theta,r,\epsilon) & \text{if } \pi \leq \theta \leq 2\pi,
\end{cases}
\end{equation}
where 
\begin{equation}\label{expansion}
F^\pm(\theta,r,\epsilon) = \sum_{i=1}^{k} \epsilon^i F_i^\pm(\theta,r) + \epsilon^{k+1} R^\pm(\theta,r,\epsilon),
\end{equation}
with $\theta \in \mathbb{S}^1$, $r > 0$, and $\epsilon > 0$ sufficiently small. Complete proofs can be found in \cite{MR3729598}.

Define the auxiliary functions
\begin{equation}\label{M_functions}
M_j^\pm(r) = \frac{1}{j!} y_j(\pm\pi,r), \quad j = 1,2,
\end{equation}
where
\begin{align*}
y_1^\pm(t,r) &= \int_0^{t} F_1^\pm(\theta,r) d\theta, \\
y_2^\pm(t,r) &= \int_0^{t} \left[2F_2^\pm(\theta,r) + 2\partial_rF_1^\pm(\theta,r)y_1^\pm(\theta,r)\right] d\theta.
\end{align*}

The \emph{averaged function of order $j$} is given by:
\begin{equation}
M_j(r) = M_j^+(r) - M_j^-(r).
\end{equation}

\begin{proposition}\label{prop3} 
Let $M_l$ ($l \in \{1,2\}$) be the first non-vanishing averaged function. For $\epsilon$ sufficiently small, each simple zero $r_0$ of $M_l$ corresponds to a hyperbolic limit cycle of system \eqref{polar_sys} that approaches $r_0$ as $\epsilon \to 0$.
\end{proposition}
The proof appears in \cite[Theorem 1]{MR3729598}.

\subsection{Lyapunov Approach for Piecewise Holomorphic Systems}\label{lya_metd}

We summarize the Lyapunov method for analyzing limit cycles in piecewise holomorphic systems. Consider the system
\begin{equation}\label{pwhs_eq_1}
\begin{cases}
\dot{w}=G^{+}(w) & \text{when } \operatorname{Im}(w)\geq0, \\ 
\dot{w}=G^{-}(w) & \text{when } \operatorname{Im}(w)\leq0,
\end{cases}
\end{equation}
where $G^\pm(w)=\sum_{k=1}^\infty G_k^\pm(w)$ are holomorphic functions with $G_1^\pm(w)=(i+\lambda^\pm)w$. To analyze the existence of limit cycles for system \eqref{pwhs_eq_1}, it suffices to study the zeros of the displacement function $\Delta_1(r) = (g^-)^{-1}(r) - g^+(r)$, where $g^\pm$ are the half-return maps associated with $\dot{w} = G^\pm(w)$. These maps are locally well-defined when the origin is a weak focus for both $\dot{w} = G^+(w)$ and $\dot{w} = G^-(w)$, and lies on the discontinuity line.
\begin{definition}
The origin is a \emph{weak focus of order $k$} for \eqref{pwhs_eq_1} if the displacement function $\Delta(r)=g^+(g^-(r))-r$ satisfies
$$\Delta(r)=V_k r^k + \mathcal{O}(r^{k+1}), \quad V_k\neq 0.$$
The coefficient $V_k$ is called the \emph{$k$-th Lyapunov quantity}.
\end{definition}

The following results, proved in \cite[Theorem D]{GASRONSIL1}, provide tools for analyzing limit cycles.

\begin{theorem}[Lyapunov Quantities]\label{teob}
For system \eqref{pwhs_eq_1} the first five Lyapunov quantities $V_1,\ldots,V_5$ can be computed as:
\begin{itemize}
\item[(i)] $V_1=e^{(\lambda^++\lambda^-)\pi}-1$; \vspace{0.3cm}
\item[(ii)] $V_2=\omega_2^+(\pi)+\omega_2^-(\pi)e^{3\lambda^+\pi}$;\vspace{0.3cm}
\item[(iii)] $V_3=e^{\lambda^+\pi}\omega_3^+(\pi)-2(\omega_2^+(\pi))^2+\omega_3^-(\pi)e^{5\lambda^+\pi}$;\vspace{0.3cm}
\item[(iv)] $V_4=e^{2\lambda^+\pi}\omega_4^+(\pi)-5e^{\lambda^+\pi}\omega_2^+(\pi)\omega_3^+(\pi)+5(\omega_2^+(\pi))^3+\omega_4^-(\pi)e^{7\lambda^+\pi};$\vspace{0.3cm}
 \item[(v)] $\!\!\!\begin{array}{rcl}
 V_5\!\!\!\!&=&\!\!\!\!e^{3\lambda^+\pi}\omega_5^+(\pi)+21e^{\lambda^+\pi}(\omega_2^+(\pi))^2\omega_3^+(\pi)-14(\omega_2^+(\pi))^4\vspace{0.3cm}\\
 & &\!\!\!\!-3e^{2\lambda^+\pi}(\omega_3^+(\pi))^2-6e^{2\lambda^+\pi}\omega_2^+(\pi)\omega_4^+(\pi)+\omega_5^-(\pi)e^{9\lambda^+\pi};
           \end{array}$
\end{itemize}
where $\omega_i^\pm(\pi)$ are explicit functions of the coefficients of $F^\pm$ (full expressions in \cite{GASRONSIL1}).
\end{theorem}

We conclude this section by establishing that if the piecewise holomorphic system \eqref{pwhs_eq_1} has a weak focus of order $k$ at the origin, then any sufficiently small perturbation preserving the origin as an equilibrium point can produce at most $k-1$ limit cycles. Furthermore we provide computable conditions to determine when exactly $k-1$ limit cycles emerge. Consider the perturbed system
\begin{equation}\label{pwhs_eq_per_1}
\begin{cases}
\dot{w} = \widetilde{G}^+(w,s), & \text{for } \operatorname{Im}(w) \geq 0, \\
\dot{w} =  \widetilde{G}^-(w,s), & \text{for } \operatorname{Im}(w) \leq 0,
\end{cases}
\end{equation}
where $ \widetilde{G}^\pm(w,0) = G^\pm(w)$ and $s = (s_1,\ldots,s_{k-1}) \in \mathbb{R}^{k-1}$ are perturbation parameters. Let $W_j(s)$ denote the corresponding Lyapunov quantities, with $W_j(0) = V_j$ from Theorem~\ref{teob}.
 \begin{proposition}\label{aux_lemma} 
 	If piecewise holomorphic system  \eqref{pwhs_eq_per_1}, with $s=0$ has a weak focus order $k$ at $p=(0,0),$ then at most $k-1$ limit cycles (counting their multiplicities) bifurcate from $p.$ 
 	Moreover, if $\det(J_{s}({W_1},\dots,{W_{k-1}}))(0)\neq 0$\footnote{Here $J_s$ denotes the Jacobian matrix of $(W_1,\dots,W_{k-1})$ with respect to the parameters $s=(s_1,\dots,s_{k-1})$.}, then there exist real parameters $s_j,$ $j=1,\dots,k-1$ small enough such it  has $k-1$ hyperbolic limit cycles bifurcating from the origin.   
 \end{proposition}
This result follows from \cite[Proposition 4]{GASRONSIL1}.

\subsection{Local Integrability of Holomorphic Systems}\label{sec7}

This subsection presents preliminary results on normal forms and integrability for holomorphic systems. We begin by introducing \emph{conformal conjugacy}, a key tool for classifying holomorphic functions near singularities, and then derive first integrals for associated differential equations.

\subsubsection{Conformal Conjugacy and Normal Forms}

The following definition establishes equivalence between holomorphic functions via conformal mappings.

\begin{definition}  
Let \( f \) and \( g \) be holomorphic in a punctured neighborhood of \( 0 \in \mathbb{C} \). We say \( f \) and \( g \) are \emph{\(0\)-conformally conjugated} if there exist \( R > 0 \) and a conformal map \( \Phi \colon D(0,R) \to D(0,R) \)\footnote{Here $D(0,R) = \{w \in \mathbb{C} : |w| < R\}$ denotes the open disk of radius $R$ centered at $0$.} such that:  
\begin{itemize}  
    \item[(i)] \( \Phi(0) = 0 \), and  
    \item[(ii)] \( \Phi \circ \varphi_f(t,w) = \varphi_g(t,\Phi(w)) \) for all \( w \in D(0,R) \setminus \{0\} \) and \( t \) where both flows \( \varphi_f, \varphi_g \) are defined in \( D(0,R) \).  
\end{itemize}  
More generally, let \( f \) and \( g \) be holomorphic functions defined in some punctured neighborhoods of \( w_1 \in \mathbb{C} \) and \( w_2 \in \mathbb{C} \), respectively. We say that \( f \) and \( g \) are \emph{\(w_1w_2\)–conformally conjugated} if \( f \circ (w - w_1) \) and \( g \circ (w - w_2) \) are \(0\)-conformally conjugated.
\end{definition}  

The next proposition classifies holomorphic functions near zeros/poles up to conformal conjugacy \cite{BT,GGJ2}.

\begin{proposition}[Normal Forms]\label{GGJ}  
Let \( f \) be holomorphic near \( w_0 \in \mathbb{C} \). Then \( f \) is conformally conjugate at \( w_0 \) to:  
\begin{itemize}  
    \item[(a)] \( g(w) \equiv 1 \) if \( f(w_0) \neq 0 \),  
    \item[(b)] \( g(w) \equiv f'(w_0)w \) if \( f \) has a simple zero at \( w_0 \),  
    \item[(c)] \( g(w) \equiv \frac{ w^n}{1 + \gamma w^{n-1}} \) if \( w_0 \) is a zero of order \( n > 1 \) and \( \operatorname{Res}(1/f, w_0) = \gamma\in\mathbb{C} \),  
    \item[(d)] \( g(w) \equiv w^{-n} \) if \( w_0 \) is a pole of order \( n \).  
\end{itemize}  
\end{proposition}  

\subsubsection{Integrability and First Integrals}

For integrability we restrict to star-shaped domains.

\begin{definition}  
A domain \( \Omega \subseteq \mathbb{C} \) is \emph{star-shaped} with respect to \( w_0 \in \Omega \) if the segment with endpoints $w_0$ and $w$ is contained in \( \Omega \) for all \( w \in \Omega \).  
\end{definition}  

The next standard result appears in most complex analysis textbooks, including \cite{Conway}.

\begin{lemma} 
If \( \Omega \) is star-shaped and \( f \colon \Omega \to \mathbb{C} \) is holomorphic, then \( f \) admits a primitive in \( \Omega \).  
\end{lemma}  

Now consider \( f \) holomorphic in \( D = D(0,R) \setminus \{0\} \), with \( 0 \) as its only zero or isolated singularity. Removing a ray \( L \) yields a star-shaped domain \( \Omega = D \setminus L \), where \( 1/f \) admits a primitive \( G \). For holomorphic vector fields we have the following characterization of trajectories, whose proof appears in \cite{MR4537508}.

\begin{theorem}\label{teoint}  
Let \( G \) be a primitive of \( 1/f \) in \( \Omega \). Then the trajectories of \( \dot{w} = f(w) \) in \( \Omega \) coincide with the level curves of \( \operatorname{Im}( G(w) )\).  
\end{theorem}  

\subsection{A miscellany of results}

The following lemma will be useful to establish the existence of simple zeros of the averaged functions. Its proof can be found in \cite[Lemma 4.5]{MR2170413}.

\begin{lemma}\label{lemma1}
	Consider $r+1$ linearly independent functions $f_i:U\subset\mathbb{R}\rightarrow\mathbb{R},$ $i=0,\dots,r$
	\begin{itemize}
		\item[(a)] Given $r+1$ arbitrary values $x_i\in U,$ $i=0,\dots,r$ there exist $r+1$ constants $C_i,$ $i=0,\dots,r$ such that
		\begin{equation}\label{eq:lemma}
			f(x):=\sum_{i=0}^r C_if_i(x)
		\end{equation}
		is not the zero function and $f(x_i)=0,$ $i=0,\dots,r.$
		\item[(b)] Furthermore, if all $f_i$ are analytic functions on $U$ and there exists $ 0\le j \le r$ such that $f_j|_U$ has a constant sign, it is possible to get an $f$ given by \eqref{eq:lemma}, such that it has at least $r$ simple zeros in $U.$
	\end{itemize}
\end{lemma}
 
Another classical tool that we shall employ is \textit{Descartes Theorem}, which provides information about the number of positive real zeros of a polynomial based on the sign pattern of its coefficients. For a detailed exposition we refer the reader to \cite{MR0174165}. Recall that for an ordered list of $r+1$ non-zero real numbers $[a_0,a_1,\ldots,a_r]$, the number of sign variations $m$ ($0\le m\le r$) is defined as the number of indices $j\le r-1$ such that $a_j a_{j+1}<0$.

\begin{theorem}[Descartes Theorem]\label{descartes}
	Consider the real polynomial $P(x)=a_{0}x^{i_0}+\dots+a_{p}x^{i_r}$ with $0\leq i_0<\dots<i_r$ and  $a_{j}$ non-zero real constants for $j\in\{0,\ldots,r\}.$ If the number of sign variations  of $[a_0,a_1,\ldots,a_r]$  is $m$, then $P(x)$ has exactly   $m-2n$ positive real zeros counting their multiplicities, where $n$ is a non negative integer number.

\end{theorem}

\section{M\"obius Transformations and the Equivalence Between Circular and Linear Switching Manifolds}\label{sec:Mobius}

To explore the behavior near the switching manifold and to exploit the symmetries of the unit circle, it is convenient to consider Möbius transformations that preserve $\mathbb{S}^1$ (or map it to a line). These transformations allow to translate the problem to a more convenient coordinate system while preserving the piecewise holomorphic structure. In particular, we need to understand how the vector field transforms under such a change of coordinates. The following fundamental lemma provides the explicit expression for the pushforward of a PWCS under a general Möbius transformation.

\begin{lemma}[Fundamental Lemma]\label{pushforward}
   Given the PWCS \eqref{ch4:eq111} and the Möbius transformation \eqref{mobius_map}, the pushforward of the differential system \(\dot{z} = F^\pm(z)\) under the Möbius transformation is given by
\begin{equation}\label{VFPF}
\dot{w} = (\phi_*F^\pm)(\phi^{-1}(w))= \frac{(cw-a)^2}{ad-bc}F^\pm\left(\frac{-dw + b}{cw - a}\right),
\end{equation}
where $\phi_*$ denotes the usual pushforward, \(w = \phi(z)\) and \(\phi^{-1}(w)\) is the inverse Möbius transformation.
\end{lemma}
\begin{proof} Consider the Möbius Transformation
$   w=\phi(z) = (az + b)/(cz + d)$
where \(a, b, c,\) and \(d\) are constants with \(ad - bc \neq 0\). Notice that
   \begin{equation}\label{Mob_transf}
   \phi^{-1}(w) = \frac{-dw + b}{cw - a}\quad\text{and} \quad\phi'(z) = \frac{ad - bc}{(cz + d)^2}.
   \end{equation}
Applying the chain rule to the equation $w=\phi(z)$, we have that
   $\dot{w} = \frac{d}{dt} \phi(z) = \phi'(z) \dot{z}$.
   Since \(\dot{z} = F^\pm(z)\), we obtain
   \begin{equation}\label{eq_Fz}
   \dot{w} = \phi'(z) F^\pm(z)
   \end{equation}
  Substituting \(z = \phi^{-1}(w)\) into \eqref{eq_Fz} gives \(\dot{w} = \phi'(\phi^{-1}(w)) F^\pm(\phi^{-1}(w))\). The result then follows from \eqref{Mob_transf}.
\end{proof} 

The following corollary shows an important consequence: normal forms of holomorphic singularities are preserved under Möbius transformations. This means that if we understand the local behavior of a system near a singularity in the straight-line case, we automatically understand it in the circular case after applying the appropriate transformation.
\begin{corollary}\label{mobius-normal}
    Let \( \phi(z)\) be a Möbius transformation in \eqref{mobius_map} and let $F$ be a holomorphic function. If $F(z)$ is in one of the normal forms listed in Proposition~\ref{GGJ} (centered at $z=0$), 
then the pushforward $G(w)=(\phi_*F)(\phi^{-1}(w))$ is in the same normal form centered at $w=\phi(0)=b/d$.
\end{corollary}
\begin{proof}
First, suppose that $F(z)\equiv 1,$ then $G(w)=\frac{(cw-a)^2}{D},$ where $D=ad-bc.$ Thus, $G(b/d)=D/d^2\neq0.$ From Proposition \ref{GGJ} we have that $G$ and $g(w)\equiv 1$ are $\phi(0)0-$conformally conjugated.

  In the case where \( F \) is $(a+ib)z$, then $G(w)=\frac{(cw-a)(-dw+b)}{D^2}.$ Thus, $G(b/d)=0$ and $G'(b/d)=a+ib.$ Proposition \ref{GGJ} implies that $G$ and $g(w)\equiv (a+ib)w$ are $\phi(0)0-$conformally. 

    Now suppose that $F(z)\equiv z^n,$ then $$G(w)=\frac{(cw-a)^{2-n}(-dw+b)^n}{D}.$$  If $n\in\mathbb{Z}_{\geq 2},$ then $G^{(k)}(b/d)=0,$ for $k=0,\dots,n-1$ and $G^{(n)}(b/d)=n!d^{2(n-1)}/D^{n-1}\neq0.$ In addition the expansion in Laurent series of $1/G(w)$ around $b/d$ is 
    $$\frac{1}{G(w)}=\frac{D^{n-1}}{d^{2(n-1)}(w-b/d)^n}-\dots+\frac{(-c)^{n-2}D}{d^{n}(w-b/d)^2}+\mathcal{O}(w-b/d)^2.$$
    This implies that $Res(1/G,b/d)=0.$ By Proposition \ref{GGJ} we can conclude that $G$ and $g(w)\equiv w^n$ are $\phi(0)0-$conformally conjugated.

    Otherwise, if $n\in\mathbb{Z}_{\leq -1},$ then the expansion in Laurent series of $G(w)$ around $b/d$ is 
    $$G(w)=\frac{D^{n+1}}{d^{2(n+1)}(w-b/d)^n}+\mathcal{O}(w-b/d)^{1-n}.$$ Thus \( G \) has a pole of order \( n \) at $b/d$. By Proposition \ref{GGJ} we have that \( G \) and \( 1/w^n \) are $\phi(0)0-$conformally conjugated.

    Finally assume that $F(z)\equiv \frac{\gamma z^n}{1+z^{n-1}},$ then $$G(w)=\frac{\gamma(cw-a)(-dw+b)^n}{D((-aw+c)^{n-1}+(-dw+b)^{n-1})}.$$  Since $n\in\mathbb{Z}_{\geq 2},$ then $G^{(k)}(b/d)=0$ for $k=0,\dots,n-1$ and $G^{(n)}(b/d)=n!d^{2(n-1)}/D^{n-1}\neq0.$ In addition, the expansion in Laurent series of $1/G(w)$ around $b/d$ is 
    $$\frac{1}{G(w)}=\frac{D^{n-1}}{d^{2(n-1)}(w-b/d)^n}-\dots+\frac{1}{\gamma(w-b/d)}+\frac{cd}{D\gamma}+\mathcal{O}(w-b/d).$$
    This implies that $Res(1/G,b/d)=1/\gamma.$ By Proposition \ref{GGJ} we can conclude that $G$ and $g(w)\equiv \frac{\gamma w^n}{1+w^{n-1}}$ are $\phi(0)0-$conformally conjugated.
\end{proof}

The Fundamental Lemma \ref{pushforward} provides an explicit formula describing how a holomorphic vector field transforms under a Möbius transformation. This result underpins our approach, as it enables to reduce the study of limit cycles in systems with circular discontinuities to the more tractable setting of straight-line discontinuities. 

Applying the transformation \eqref{mobius_map1} which maps $\mathbb{S}^1$ onto the real line (see Figure \ref{fig_conf_map}), the pushforward formula \eqref{VFPF} gives precise control over the degrees and coefficients of the transformed vector fields. Moreover $\phi$ maps the interior (respectively exterior) of the circle onto the lower half-plane 
\[
\phi(\Sigma^-) = \{ w \in \mathbb{C} \mid \operatorname{Im}(w) < 0 \}
\]
(respectively the upper half-plane 
\[
\phi(\Sigma^+) = \{ w \in \mathbb{C} \mid \operatorname{Im}(w) > 0 \}).
\]

\begin{figure}[h]
\begin{overpic}[scale=0.5]{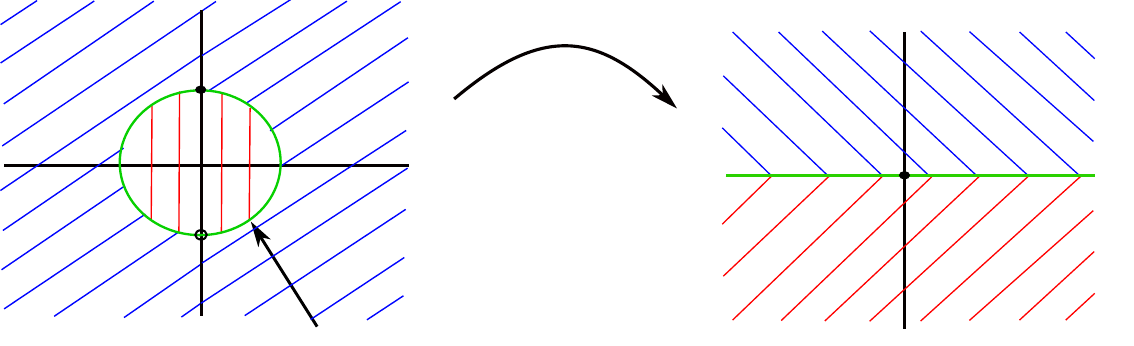}
        \put(98,15){$\phi(\Sigma)$}
        \put(25,-2){$\Sigma=\mathbb{S}^1$}
        \put(49,29){$\phi$}
\end{overpic}
\caption{\footnotesize{The conformal map $\phi(z)=\frac{-z+i}{iz-1}$.}}
\label{fig_conf_map}
\end{figure}

Consequently instead of analyzing the system with the unit circle as the discontinuity, we may equivalently study a system with the real axis as the discontinuity. In particular, the dynamics of PWHS \eqref{ch4:eq111} are equivalent to those of
\begin{equation}\label{ch4:eq222}
\begin{aligned}
\left\{
\begin{array}{l}
\dot{w}=G^{+}(w), \text{ when } \operatorname{Im}(w)\geq 0,\\[5pt]
\dot{w}=G^{-}(w), \text{ when } \operatorname{Im}(w)\leq0,
\end{array}
\right.
\end{aligned}
\end{equation}
where $G^{\pm}(w)$ denotes the pushforward of $F^\pm$ under $w=\phi(z)$, as it is given in Fundamental Lemma \ref{pushforward}. 

\begin{remark}\label{rek}
In the context of a piecewise holomorphic system where the system is defined on different regions separated by the unit circle, if a limit cycle of the system passes through a point where a Möbius transformation \( \phi \) is not defined (e.g., \( -i \) in this case), the image of this limit cycle under the Möbius transformation \( \phi \) may not be a traditional limit cycle (see Figure \ref{fig_conf_map_1}). Instead it could transform into an orbit that asymptotically approaches infinity. 

This occurs because Möbius transformations can map finite points to infinite points and vice versa. As a result the transformed orbit might not be a closed loop in the finite complex plane but could instead becomes a trajectory that closes at infinity. Thus while the original trajectory is a limit cycle, its image under \( \phi \) could be a more complex structure that does not correspond to a limit cycle in the usual sense but rather to a trajectory approaching infinity.

   \begin{figure}[h]
\begin{overpic}[scale=0.5]{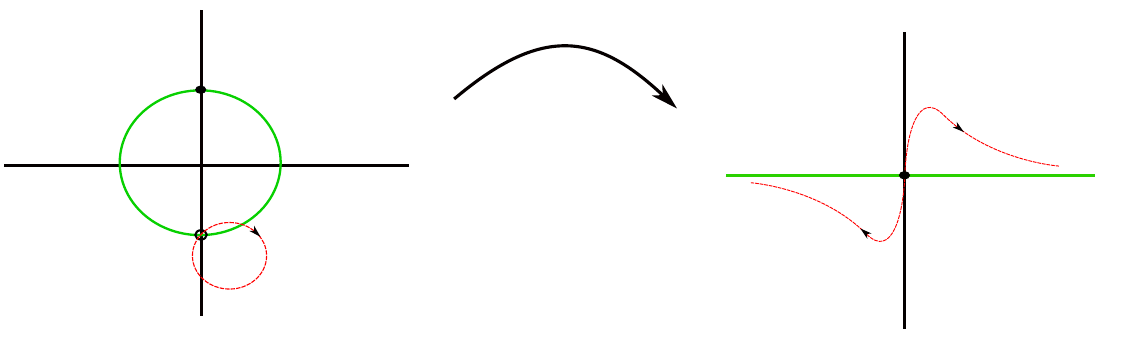}
        \put(90,20){$\phi(\gamma(t))$}
        \put(21,2){$\gamma(t)$}
        \put(49,29){$\phi$}
		\end{overpic}
\caption{\footnotesize{Figure illustrating the effect of a Möbius transformation on a limit cycle of a piecewise holomorphic system. On the left the limit cycle (red) of the original system passes through a point where the Möbius transformation \( \phi \) is not defined (e.g., \( -i \)). On the right the image of this limit cycle (red) under the transformation \( \phi \) is shown. The transformation \( \phi \) maps the original limit cycle to an orbit that does not close in the finite plane but rather approaches infinity. This demonstrates that a limit cycle passing through a singularity of the Möbius transformation may be transformed into an orbit that "closes" at infinity, rather than forming a traditional limit cycle in the image plane.}}
\label{fig_conf_map_1}
\end{figure}
\end{remark}
This remark highlights the potential change in the qualitative nature of the trajectory when subjected to a Möbius transformation, especially when the transformation has poles or is not defined at certain points.

To illustrate the practical utility of the Fundamental Lemma we construct an explicit piecewise holomorphic system with $\mathbb{S}^1$ discontinuity that possesses a limit cycle. The construction proceeds by designing a system in the transformed coordinates (where the discontinuity is a straight line) that has a readily computable Poincaré map, and then pulling it back via $\phi^{-1}$.

\begin{example}
Consider the PWHS \eqref{ch4:eq111} with 
$$F^+(z)=\left(\frac{1+i}{2}\right) (1 + z^2) \quad\text{and}\quad F^-(z)=\frac{i}{4}  \left((-1 + 2 i) + e^{-\pi} + \frac{4}{i + z}\right) (-1 + i z)^2.$$
We claim that this system has a limit cycle. From Fundamental Lemma \ref{pushforward} we get that 
$$G^\pm(w)=\frac{(iw+1)^2}{2}F^\pm\left(\frac{w+i}{iw+1}\right).$$
Thus 
$$G^+(w)=(-1+i)w \quad\text{and}\quad G^-(w)=i \left(w-\frac{1}{2} + \frac{e^{-\pi}}{2}\right).$$
Consider $w_0\in\mathbb{R}^+.$ Recall that \[w^+(t)=\left(w_0-\frac{1}{2}+\frac{e^{-\pi}}{2}\right)e^{it}+\frac{1}{2}-\frac{e^{-\pi}}{2}\]
is a solution of $\dot{w}=i\left(w-\frac{1}{2}+\frac{e^{-\pi}}{2}\right)$ satisfying that $w^+(0)=w_0$ and $w^{+}(\pi)=1-e^{-\pi}-w_0$. In addition, 
$w^-(t)=-e^{(-1+i)t}\left(w_0-1+e^{-\pi}\right)$
is a solution of $\dot{w}=(-1+i)w$ such that $w^{-}(0)=1-e^{-\pi}-w_0$ and $w^{-}(\pi)=e^{-\pi}(w_0-1+e^{-\pi})$.

Therefore the Poincaré map around $w=w_0$ is 
$$\Pi(w)=e^{-\pi}\left(w-1+e^{-\pi}\right)$$ and $\Pi'( w_{0})=e^{-\pi}<1$. Now we must seek solutions for the equation $\Pi(w_{0})=w_{0}$. The number of roots of this equation correspond to the number of limit cycles. We have a unique solution given by $w_0=-e^{-\pi}$, thus we have a unique limit cycle $\Gamma$, see Figure \ref{fig_phase_portrait_1}. Finally using the first derivative of the Poincaré map we can conclude that $\Gamma$ is stable.

      \begin{figure}[h]
\begin{overpic}[scale=0.3]{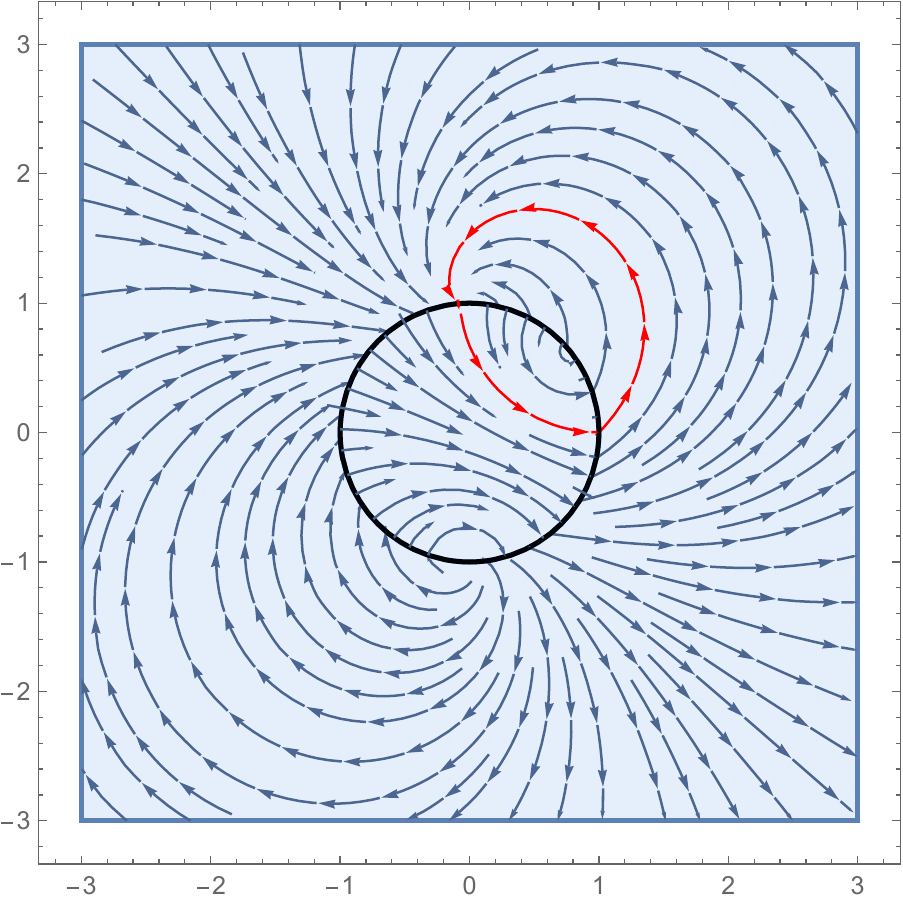}
		\end{overpic}
\caption{\footnotesize{Phase portrait of a piecewise holomorphic system with a discontinuity on the unit circle \( \mathbb{S}^1 \). The unit circle is shown in black. A limit cycle is highlighted in red, illustrating its behavior as it interacts with the discontinuity. The phase portrait demonstrates how the trajectories behave in the regions inside and outside the unit circle, emphasizing the dynamics of the limit cycle and its interaction with the discontinuity.}}
\label{fig_phase_portrait_1}
\end{figure}
\end{example}

\begin{theorem}\label{thm:mobius_limit_cycle}
Let $\phi(z)$ be the Möbius transformation defined in \eqref{mobius_map1}. 
Let $\gamma$ be a closed orbit of system \eqref{ch4:eq111} such that:
\begin{itemize}
\item[(i)] it intersects the switching manifold only at crossing points (in the Filippov sense); and
\item[(ii)] it does not pass through the pole $z=-i$ of $\phi$.
\end{itemize}
Then $\gamma$ is a limit cycle of system \eqref{ch4:eq111} if, and only if, $\phi(\gamma)$ is a limit cycle of system \eqref{ch4:eq222} (see Figure \ref{fig_conf_map_2}). Moreover:
\begin{itemize}
\item[(a)] The period is preserved.
\item[(b)] The stability type (attracting, repelling, hyperbolic) is preserved.
\end{itemize}
\end{theorem}

\begin{proof}
Let $\gamma=\{z(t):t\in[0,T)\}$ be a $T$-periodic orbit of \eqref{ch4:eq111} satisfying the hypotheses. Condition (ii) guarantees that $\gamma$ avoids the pole $z=-i$, so $\phi$ is a diffeomorphism in a neighborhood of $\gamma$. Define $w(t)=\phi(z(t))$.

By the Fundamental Lemma \ref{pushforward} the vector fields $F^\pm$ and $G^\pm$ are conjugated by $\phi$ in each region. Since $\gamma$ intersects $\Sigma$ only at crossing points, the trajectory evolves piecewise smoothly and $w(t)$ satisfies system \eqref{ch4:eq222} for all $t$.

From $z(t+T)=z(t)$ it follows that
$$
w(t+T)=\phi(z(t+T))=\phi(z(t))=w(t),
$$
so $\phi(\gamma)$ is $T$-periodic, proving (a).

Assume that $\gamma$ is a limit cycle. To prove $\phi(\gamma)$ is also a limit cycle we must show that it is an isolated periodic orbit. Suppose, for contradiction, that there exists a sequence $\{\widetilde{\gamma}_n\}$ of distinct periodic orbits of \eqref{ch4:eq222} converging to $\phi(\gamma)$ in the Hausdorff distance. Since $\phi(\gamma)$ is compact and $\phi$ is a diffeomorphism in a neighborhood of $\gamma$, the inverse $\phi^{-1}$ is continuous on $\phi(\gamma)$. For sufficiently large $n$ the orbits $\widetilde{\gamma}_n$ lie within this neighborhood, so applying $\phi^{-1}$ yields a sequence $\{\gamma_n=\phi^{-1}(\widetilde{\gamma}_n)\}$ of distinct periodic orbits of \eqref{ch4:eq111} converging to $\gamma$, contradicting the isolation of $\gamma$. Hence $\phi(\gamma)$ is isolated. The converse follows by applying the same argument with $\phi^{-1}$.

To compare stability, let $z_0\in\gamma$ be a point away from the switching manifold $\Sigma=\mathbb{S}^1$ and from the pole $z=-i$. Since $\gamma$ intersects $\Sigma$ only at crossing points and otherwise lies entirely in either $|z|<1$ or $|z|>1$, we can choose $z_0$ in the interior of one of these regions. Let $\Sigma_0$ be a smooth transversal section through $z_0$, for instance a small line segment orthogonal to $\gamma$ at $z_0$, chosen sufficiently small so that $\Sigma_0$ is contained entirely within the same region and such that the Poincaré map $P:\Sigma_0\to\Sigma_0$ associated with $\gamma$ is well defined and smooth.

Since $\phi$ is a diffeomorphism in a neighborhood of $\gamma$, its derivative $d\phi_{z_0}$ is an invertible linear map. Transversality of $\Sigma_0$ to $\gamma$ at $z_0$ means that $T_{z_0}\Sigma_0$ and $T_{z_0}\gamma$ span $\mathbb{R}^2$. Applying $d\phi_{z_0}$, which is an isomorphism, we obtain that $T_{w_0}\widetilde{\Sigma}_0$ and $T_{w_0}\widetilde{\gamma}$ also span $\mathbb{R}^2$, where $w_0=\phi(z_0)$, $\widetilde{\Sigma}_0=\phi(\Sigma_0)$, and $\widetilde{\gamma}=\phi(\gamma)$. Hence $\widetilde{\Sigma}_0$ is transverse to $\widetilde{\gamma}$ at $w_0$, and we denote by $\widetilde{P}:\widetilde{\Sigma}_0\to\widetilde{\Sigma}_0$ the corresponding Poincaré map of system \eqref{ch4:eq222}.

Since the flow in each region is conjugated by $\phi$, and the sections are chosen so that orbits intersect them only within the same region before returning, we have the commutative relation $\widetilde{P}=\phi\circ P\circ\phi^{-1}$. Differentiating at the fixed points yields
$$
\widetilde{P}'(w_0)=\phi'(z_0)P'(z_0)(\phi^{-1})'(w_0).
$$
Since $(\phi^{-1})'(w_0)=1/\phi'(z_0)\neq0$, we obtain $\widetilde{P}'(w_0)=P'(z_0)$. Hence the stability type (attracting, repelling, or hyperbolic) is preserved. This establishes (b).
\end{proof}

   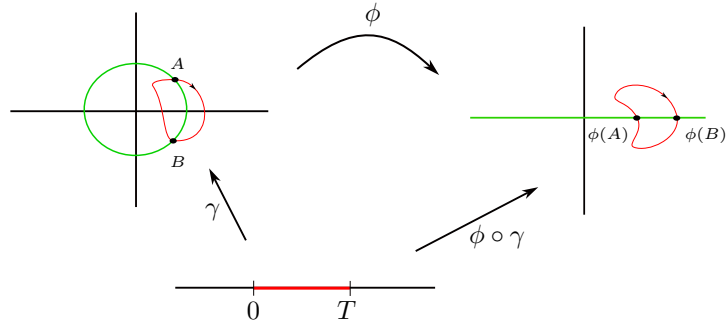
\begin{figure}[h]
\begin{overpic}[scale=0.5]{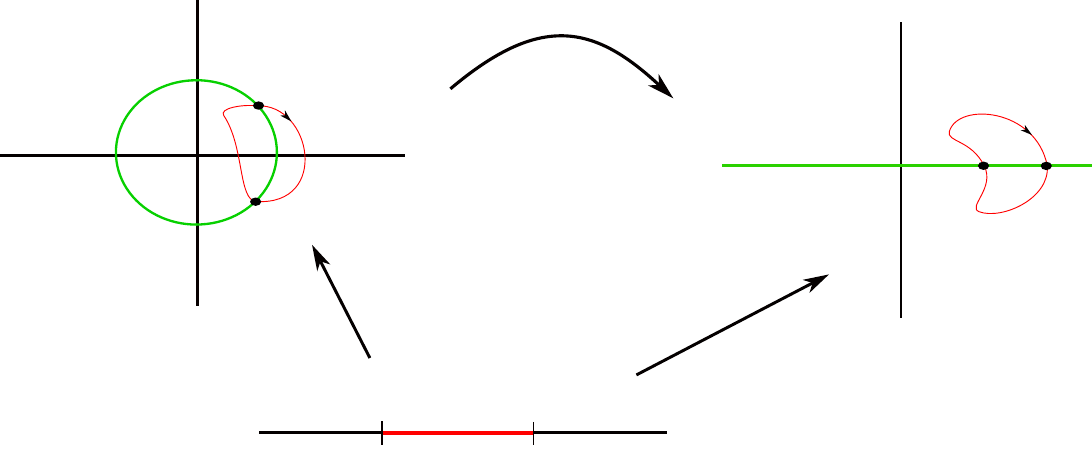}
        \put(66,9){$\phi\circ\gamma$}
        \put(28,13){$\gamma$}
         \put(23,34){{\tiny$A$}}
          \put(23,20){{\tiny$B$}}
        \put(51,41){$\phi$}
        \put(34,-2){$0$}
        \put(47,-2){$T$}
         \put(83,24){{\tiny$\phi(A)$}}
          \put(97,24){{\tiny $\phi(B)$}}
		\end{overpic}
\caption{\footnotesize{This figure illustrates the effect of Möbius transformations on periodic orbits in the complex plane. Both the original periodic orbit and the transformed orbit are shown in red. Möbius transformations not only map periodic orbits to other periodic orbits but also preserve their period.}}
\label{fig_conf_map_2}
\end{figure}

\begin{remark}
Theorem \ref{thm:mobius_limit_cycle} provides a complete correspondence between the dynamics of piecewise holomorphic systems with circular discontinuity and those with linear discontinuity, provided the limit cycles avoid the pole of the Möbius transformation. In particular, any lower bound on the number of limit cycles obtained for systems with a straight-line discontinuity (via averaging theory, Lyapunov quantities, or explicit construction) automatically yields the same lower bound for systems with $\mathbb{S}^1$ discontinuity. This observation underlies all the subsequent results in this paper.
\end{remark}

A natural question that arises from the correspondence established in Theorem \ref{thm:mobius_limit_cycle} is how the algebraic structure of limit cycles behaves under Möbius transformations. Since many of the explicit constructions in the literature (see, for instance, \cite{doi:10.1142/S0218127418500396}) yield limit cycles that are algebraic curves, it is important to understand whether this property is preserved when transferring examples between the straight-line and circular settings. The following results show that the answer is affirmative, provided we take into account the effect of the transformation on the degree of the defining polynomials and the special care required when the limit cycle passes through the pole of the Möbius map.

A fundamental tool in our analysis is the action of Möbius transformations on algebraic curves and on the dynamics. Since Möbius transformations are birational maps of degree $1$, they preserve the algebraic nature of curves, although the degree may change. The following proposition establishes that the property of being an algebraic limit cycle is invariant under such transformations.

\begin{proposition}\label{prop:alg_mobius}
Let $\phi$ be a Möbius transformation of the form \eqref{mobius_map} and consider the PWHS \eqref{ch4:eq111} with switching manifold $\mathbb{S}^1$. Suppose $\gamma$ is a closed orbit that satisfies the hypotheses of Theorem \ref{thm:mobius_limit_cycle}. Then $\gamma$ is an algebraic limit cycle of the original system if, and only if, $\phi(\gamma)$ is an algebraic limit cycle of the transformed system, whose switching manifold is $\phi(\mathbb{S}^1)$ (either a circle or a line).
\end{proposition}

\begin{proof}
Möbius transformations are rational maps of degree $1$, hence they send algebraic curves to algebraic curves. More precisely, if $\Gamma = \{P(x,y)=0\}$ is an algebraic curve of degree $d$, then $\phi(\Gamma)$ is an algebraic curve. Indeed,
\[
\phi(\Gamma) = \{ (u,v) \in \mathbb{R}^2 \mid P(\phi^{-1}(u,v)) = 0 \},
\]
and clearing denominators in $P(\phi^{-1}(u,v))$ yields a polynomial $Q(u,v)$ that defines $\phi(\Gamma)$ (up to possible extraneous factors). The degree of $Q$ is at most $2d$, since $\phi^{-1}$ consists of rational functions with numerator and denominator of degree $1$.

Now suppose $\gamma$ is an algebraic limit cycle of the original system. By Definition there exist irreducible polynomials $p^+$ and $p^-$ such that
\[
\gamma \cap (\Sigma^+ \setminus \Sigma) \subset \{p^+=0\}, \qquad
\gamma \cap (\Sigma^- \setminus \Sigma) \subset \{p^-=0\}.
\]
Applying $\phi$ we obtain
\[
\phi(\gamma) \cap \big(\phi(\Sigma^+) \setminus \phi(\Sigma)\big) \subset \phi(\{p^+=0\}), \qquad
\phi(\gamma) \cap \big(\phi(\Sigma^-) \setminus \phi(\Sigma)\big) \subset \phi(\{p^-=0\}).
\]
As argued above $\phi(\{p^\pm=0\})$ are algebraic curves. By Theorem \ref{thm:mobius_limit_cycle} $\phi(\gamma)$ is a limit cycle of the transformed system. Therefore $\phi(\gamma)$ is an algebraic limit cycle.

Conversely, if $\phi(\gamma)$ is an algebraic limit cycle, applying the same argument with $\phi^{-1}$ (which is also a Möbius transformation) and using Theorem \ref{thm:mobius_limit_cycle} in the reverse direction shows that $\gamma$ is algebraic.
\end{proof}

\begin{remark}\label{rem:bidegree}
Proposition \ref{prop:alg_mobius} shows that the property of being an algebraic limit cycle is preserved under Möbius transformations. However, the bidegree $(m,n)$ is not necessarily invariant. For instance, if the original cycle $\gamma = \{p^-, p^+\}$ has $p^\pm$ of degree $m^\pm$, the transformed cycle $\phi(\gamma) = \{q^-, q^+\}$ satisfies
\[
\deg(q^\pm) \leq 2\deg(p^\pm),
\]
with equality generically. This upper bound arises because $\phi^{-1}$ introduces rational functions with denominator of degree $1$; clearing denominators may double the degree. In particular, when $\phi$ maps $\mathbb{S}^1$ to $\mathbb{R}$ (as in \eqref{mobius_map1}), the bidegree of the transformed cycle provides a measure of the complexity of the original cycle in the canonical straight-line setting.
\end{remark}

\begin{remark}\label{rem:mobius_pole}
The hypothesis in Proposition \ref{prop:alg_mobius} that the limit cycle does not pass through the pole of the M\"obius transformation is essential to guarantee that the image remains a bounded algebraic curve. To illustrate this consider the system
\[
\dot{w} =
\begin{cases}
iw, & \operatorname{Im}(w) \geq 0, \\[6pt]
\left(\dfrac{3}{2} + \dfrac{i}{3}\right) + i w - \left(\dfrac{3}{2} - \dfrac{i}{3}\right) w^2, & \operatorname{Im}(w) \leq 0,
\end{cases}
\]
which, following the same ideas used in the proof of Theorem \ref{thm:alg_s1}, can be shown to have the algebraic limit cycle $\Gamma = \{w\in\mathbb{C}: |w| = 1 \}$. This cycle passes through the pole $z = i$ of the M\"obius transformation $\phi^{-1}(w) = (w+i)/(iw+1)$. Applying $\phi^{-1}$ to the system yields a piecewise holomorphic system with switching manifold $\mathbb{S}^1$ whose limit cycle is $\phi^{-1}(\Gamma) = \mathbb{R}$, the real line. Although $\mathbb{R}$ is an algebraic curve (a line) it is unbounded and passes through infinity. Thus the hypothesis that the cycle avoids the pole is necessary to ensure that the transformed cycle is a bounded algebraic curve in the finite complex plane.
\end{remark}

\section{Proof of Main Results}\label{sec:proofs}

\subsection{Proof Theorem \ref{teo1}}
Consider the piecewise holomorphic system \eqref{ch4:eq111} with
\[
F^\pm(z)=\frac{1+z^2}{2}+\varepsilon h^\pm\!\bigl(\phi(z)\bigr),
\]
where \(\phi(z)=(-z+i)/(iz-1)\) and \(h^\pm(w)=\sum_{k=0}^{n^\pm}(a_k^\pm+ib_k^\pm)w^k\) are holomorphic polynomials of degrees \(n^\pm\ge1\).  
Applying the push‑forward formula \eqref{VFPF} (Fundamental Lemma~\ref{pushforward}) we obtain after a direct computation
\[
G^\pm(w)=iw+\varepsilon\sum_{k=0}^{n^\pm+2}\bigl(p_k^\pm+iq_k^\pm\bigr)w^k,
\]
with coefficients given by
\[
p_k^\pm=\operatorname{Re}\!\Bigl(-\tfrac{\delta_{k-2}^\pm}{2}+i\delta_{k-1}^\pm+\tfrac{\delta_k^\pm}{2}\Bigr),\qquad
q_k^\pm=\operatorname{Im}\!\Bigl(-\tfrac{\delta_{k-2}^\pm}{2}+i\delta_{k-1}^\pm+\tfrac{\delta_k^\pm}{2}\Bigr),
\]
where \(\delta_k^\pm=a_k^\pm+ib_k^\pm\) for \(k=0,\dots,n^\pm\) and \(\delta_k^\pm=0\) for indices outside this range.

Passing to polar coordinates \(w=re^{i\theta}\) and writing the differential equation for \(dr/d\theta\) we arrive at
\[
\frac{dr}{d\theta}=F^\pm(r,\theta,\varepsilon)=
\varepsilon F_1^\pm(\theta,r)+\varepsilon^2 F_2^\pm(\theta,r)+O(\varepsilon^3),
\]
with
\begin{align*}
F_1^\pm(\theta,r)&=p_1^\pm r+c_0^\pm+\sum_{k=1}^{n^\pm+1}r^{k+1}c_k^\pm,\\
F_2^\pm(\theta,r)&=\frac1r\Bigl(-q_1^\pm r-d_0^\pm-\sum_{k=2}^{n^\pm+2}d_{k-1}^\pm r^k\Bigr)
\Bigl(p_1^\pm r+c_0^\pm+\sum_{k=2}^{n^\pm+2}c_{k-1}^\pm r^k\Bigr),
\end{align*}
where
\begin{align*}
c_0^\pm&=p_0^\pm\cos\theta+q_0^\pm\sin\theta,\qquad 
c_k^\pm=p_{k+1}^\pm\cos(k\theta)-q_{k+1}^\pm\sin(k\theta)\;(k\ge1),\\
d_0^\pm&=q_0^\pm\cos\theta-p_0^\pm\sin\theta,\qquad 
d_k^\pm=q_{k+1}^\pm\cos(k\theta)+p_{k+1}^\pm\sin(k\theta)\;(k\ge1).
\end{align*}

The first‑order averaged function is
\[
M_1^\pm(r)=\int_0^{\pm\pi}F_1^\pm(\theta,r)\,d\theta
=2q_0^\pm\pm\pi p_1^\pm r-2\sum_{k=1}^{\lfloor n^\pm/2\rfloor+1}\frac{q_{2k}^\pm}{2k-1}\,r^{2k}.
\]

Assume without loss of generality that \(n^+\ge n^-\). Then
\[
\begin{aligned}
M_1(r)
&= M_1^+(r)-M_1^-(r) \\
&= 2(q_0^+-q_0^-)+\pi(p_1^++p_1^-)r
+2\sum_{k=1}^{\lfloor n^-/2\rfloor+1}\frac{q_{2k}^--q_{2k}^+}{2k-1}\,r^{2k}
-2\!\!\sum_{k=\lfloor n^-/2\rfloor+2}^{\lfloor n^+/2\rfloor+1}
\!\frac{q_{2k}^+}{2k-1}\,r^{2k}.
\end{aligned}
\]

Thus \(M_1(r)\) is a real polynomial containing \(\lfloor n^+/2\rfloor+3\) monomials. By Descartes’ rule (Theorem~\ref{descartes}) it can have at most \(\lfloor n^+/2\rfloor+2\) positive zeros, counted with multiplicity.  

The map that sends the perturbation parameters \(\{a_k^\pm,b_k^\pm\}\) to the coefficients of \(M_1(r)\) is linear. Inspecting the formulas above one sees that its matrix has a block‑triangular structure and full rank, consequently every choice of the coefficients of \(M_1(r)\) can be realized by a suitable choice of the parameters \(a_k^\pm,b_k^\pm\); in particular, the coefficients are completely free. Since each monomial \(r^{2k}\) has constant sign on \((0,\infty)\), Lemma~\ref{lemma1} guarantees that we can select the parameters so that \(M_1(r)\) possesses exactly \(\lfloor n^+/2\rfloor+2\) simple positive zeros. Proposition~\ref{prop3} then yields, for sufficiently small \(|\varepsilon|\), the same number of hyperbolic limit cycles bifurcating from the periodic orbits of the unperturbed system \(\dot w=iw\).

To obtain further limit cycles we now require the first‑order averaged function to vanish identically. This forces
\[
p_1^+=-p_1^-,\qquad 
q_{2k}^+=q_{2k}^-\;(0\le k\le\lfloor n^-/2\rfloor+1),\qquad
q_{2k}^+=0\;(\lfloor n^-/2\rfloor+2\le k\le\lfloor n^+/2\rfloor+1).
\]

Under these conditions the second‑order averaged function becomes
\[
\begin{aligned}
M_2(r)&=-4p_1^-(p_0^++p_0^-)-2q_0^-(q_1^+-q_1^-)\\
&\quad+\pi\bigl[p_1^-(q_1^+-q_1^-)-2q_0^-(p_2^++p_2^-)-2q_2^-(p_0^++p_0^-)\bigr]r\\
&\quad+\sum_{k=1}^{n^-+1}\bigl(U_k^+-U_k^-\bigr)r^{2k}
+\sum_{k=n^-+2}^{n^++1}U_k^+ r^{2k}\\
&\quad+\pi p_1^-\!\sum_{k=1}^{\lfloor\frac{n^-+1}{2}\rfloor}\!(q_{2k+1}^--q_{2k+1}^+)r^{2k+1}
-\pi p_1^-\!\!\sum_{k=\lfloor\frac{n^-+3}{2}\rfloor}^{\lfloor\frac{n^++1}{2}\rfloor}\!q_{2k+1}^+ r^{2k+1},
\end{aligned}
\]
where
\[
U_k^\pm:=-\frac{2^{k+1} k!}{(2k)!}\sum_{2s+t=2k+1}
\left(\frac{(2(k-1))!}{2^{k-2}(k-1)!}\,p_{2s}^\pm p_t^\pm
+\frac{(t-2s)(2k-2)!}{(2s-1)2^{k-1}(k-1)!}\,q_{2s}^\pm q_t^\pm\right).
\]

The polynomial \(M_2(r)\) contains exactly \(\bigl\lfloor(3\max\{n^+,n^-\}+7)/2\bigr\rfloor\) monomials. Again the coefficients of these monomials can be chosen independently by adjusting the remaining free parameters of the perturbation. Hence, reasoning as before, we can arrange that \(M_2(r)\) has precisely \(\bigl\lfloor(3\max\{n^+,n^-\}+5)/2\bigr\rfloor\) simple positive zeros. Applying Proposition~\ref{prop3} once more we obtain, for small enough \(|\varepsilon|\), that many hyperbolic limit cycles bifurcate from the periodic orbits of the centre \(i\).

Since the upper bound provided by Descartes’ rule is attainable in both cases, we conclude that
\[
\mathcal{L}_{n^+,n^-}\ge\Bigl\lfloor\frac{3\max\{n^+,n^-\}+5}{2}\Bigr\rfloor,
\]
which completes the proof.

\subsection{Proof of Theorem \ref{teonew}}

As we will see the proof of Theorem \ref{teonew} will be a straightforward consequence of the propositions proven in this section.

 \begin{proposition}\label{imp_prop31}
       Consider the PWHS
       \begin{equation}\label{pw_fam1}
\begin{aligned}
\left\{\begin{array}{l}
\dot{z}=(1-i(\lambda+s))+A^{+}z, \text{ when } |z|\geq 1,\\[5pt]
\dot{z}=(1+i\lambda)+A^{-}z,\text{ when } |z|\leq1,
\end{array} \right.
\end{aligned}
\end{equation}
where $\lambda\neq0$, $A^+=\lambda+s+i$ and $A^-=-\lambda+i$. Then there exists a sufficiently small real parameter \( s \) such that there is one limit cycle bifurcating from the point \( i \).

   \end{proposition}
   \begin{proof}
   From Fundamental Lemma \ref{pushforward}, we get that the quadratic piecewise holomorphic system
          \begin{equation}\label{pw_fam11}
\begin{aligned}
\left\{\begin{array}{l}
\dot{w}=(i+\lambda+s)w+\widetilde{A}^{+}w^2, \text{ when } \operatorname{Im}(w)\geq 0,\\[5pt]
\dot{w}=(i-\lambda)w+\widetilde{A}^{-}w^2,\text{ when } \operatorname{Im}(w)\leq0,
\end{array} \right.
\end{aligned}
\end{equation}
such that $\widetilde{A}^+=-1+i(\lambda+s)$ and $\widetilde{A}^-=-1-i\lambda.$ We claim that there exists a real parameter $s$ small enough such that it has one limit cycle bifurcating from the origin. Indeed, using Theorem \ref{teob} to compute the Lyapunov quantities and expanding them around $s=0,$ we obtain the following expressions:
\begin{align*}
W_1 &= \pi s + \mathcal{O}_2(s); \\
W_2 &= 2 e^{\pi \lambda} (1 + e^{\pi \lambda}) \lambda 
       + e^{\pi \lambda} \left(1 + 4 \pi \lambda 
       + e^{\pi \lambda} \left(1 + 5 \pi \lambda\right)\right) s 
       + \mathcal{O}_2(s).
\end{align*}
       Notice that $W_1=0$ and $W_2\neq 0$ provided that $s=0.$ Moreover, 
       $$\det(J_s(W_1))(0)=\pi\neq 0.$$
       From Proposition \ref{aux_lemma} we can conclude that there exist a real parameter $s$ small enough such that system \eqref{pw_fam11} has 1 limit cycle. Since $\phi^{-1}(0)=i$, then  system \eqref{pw_fam1} has one limit cycle bifurcating from the point $i$.
   \end{proof}
\begin{proposition}\label{imp_prop32}
       Consider the PWHS
       \begin{equation}\label{pw_fam1_2}
\begin{aligned}
\left\{\begin{array}{l}
\dot{z}= \frac{1}{2} - \frac{1}{2} s_2 + i\frac{(4 - 3 s_1)}{6}  -  \left( \frac{20}{6} + i s_2 \right) z +  \left( \frac{1}{2} + \frac{1}{2} s_2 - i\frac{(16 + 3 s_1)}{6} \right) z^2, \text{ when } |z|\geq 1,\\[5pt]
\dot{z}= 1 + 2 i +  \left(-2 + i\right) z,\text{ when } |z|\leq1,
\end{array} \right.
\end{aligned}
\end{equation}
Then there exist real parameters $s_1$ and $s_2$ small enough such that it has 2 limit cycles bifurcating from the point \( i \).

   \end{proposition}
\begin{proof}
 From Fundamental Lemma \ref{pushforward} we obtain the piecewise quadratic holomorphic system
\begin{equation}\label{pw_fam11_2}
\begin{cases}
\dot{w} = (2 + s_1 + i)w + \left(s_2 -\frac{10}{3}i\right)w^2, & \text{when } \operatorname{Im}(w) \geq 0, \\[5pt]
\dot{w} = (-2 + i)w-(1+2i)w^2, & \text{when } \operatorname{Im}(w) \leq 0,
\end{cases}
\end{equation}
We claim that for sufficiently small real parameters $s_1, s_2$ this system exhibits 2 limit cycles bifurcating from the origin. 

Applying Theorem \ref{teob} to compute the Lyapunov quantities and expanding them around $s_1 = s_2 = 0$, we obtain:
\begin{align*}
W_1 &= \pi s_1 + \mathcal{O}_2(s); \\
W_2 &=  -\frac{4}{5} e^{2\pi} \left(1 + e^{2\pi}\right) s_2 - \frac{2}{75} e^{2\pi} \left(40 + 40 e^{2\pi} - 150 \pi - 75 e^{2\pi} \pi\right) s_1
+ \mathcal{O}_2(s); \\
W_3 &= \frac{1}{18} e^{4\pi} \left( -72 \left(1 + e^{2\pi}\right)^2 + 72 e^{12\pi} \left(1 + e^{2\pi}\right)^2 + 35 \left(-1 + e^{4\pi}\right) \right) + \mathcal{O}_1(s); \\
\end{align*}
Recall that $W_1 = W_2 = 0$ and $W_3 \neq 0$ when $s_1 = s_2 = 0$. In addition the Jacobian determinant of the first two Lyapunov quantities is nonvanishing:
\[
\det(J_s(W_1, W_2))(0) = -\frac{4\pi}{5} e^{2\pi} \left(1 + e^{2\pi}\right)  \neq 0.
\]
By Proposition~\ref{aux_lemma} there exist sufficiently small real parameters $s_1, s_2$ such that system~\eqref{pw_fam11_2} has 2 limit cycles. Since $\phi^{-1}(0) = i$, system~\eqref{pw_fam1_2} consequently has 2 limit cycles bifurcating from the point $i$.
\end{proof}
 \begin{proposition}\label{imp_prop3}
       Consider the PWHS
       \begin{equation}\label{pw_fam1_3}
\begin{aligned}
\left\{\begin{array}{l}
\dot{z}=\frac{1}{2} \left((1 - s_2) - i (s_1 + 2)\right) + (1 - i s_2) z + \frac{1}{2} (1 + s_2- i s_1) z^2, \text{ when } |z|\geq 1,\\[5pt]
\dot{z}=\frac{1}{2} \left(1 - i s_3\right) + (1 + s_3) z +  \frac{1}{2} (1 + i (2+s_3))z^2,\text{ when } |z|\leq1,
\end{array} \right.
\end{aligned}
\end{equation}
Then there exist real parameters $s_1$, $s_2$ and $s_3$ small enough such that it has 3 limit cycles bifurcating from the point \( i \).

   \end{proposition}
\begin{proof}
 From Fundamental Lemma \ref{pushforward} we obtain the piecewise quadratic holomorphic system
\begin{equation}\label{pw_fam11_3}
\begin{cases}
\dot{w} = (1 + s_1 + i)w + (s_2 + i)w^2, & \text{when } \operatorname{Im}(w) \geq 0, \\[5pt]
\dot{w} = (-1 + i)w + i(1 + s_3)w^2, & \text{when } \operatorname{Im}(w) \leq 0,
\end{cases}
\end{equation}
We claim that for sufficiently small real parameters $s_1, s_2, s_3$ this system exhibits 3 limit cycles bifurcating from the origin. 

Using Theorem \ref{teob} the Lyapunov quantities expanded around $s_1=s_2=s_3=0$ take the form
\begin{align*}
W_1 &= \pi s_1 + \mathcal{O}_2(s); \\
W_2 &= -e^{\pi} (1 + e^{\pi}) s_1 + e^{\pi} (1 + e^{\pi}) s_2 + \mathcal{O}_2(s); \\
W_3 &= \dfrac{1}{4} e^{2\pi} \left(-1 + e^{2\pi} + 3\pi - e^{2\pi} \pi\right) s_1 
       - \dfrac{1}{2} e^{2\pi} \left(-1 + e^{2\pi}\right) s_2 \\
     &\quad - \dfrac{1}{2} e^{2\pi} \left(-1 + e^{2\pi}\right) s_3 + \mathcal{O}_2(s); \\
W_4 &= \dfrac{1}{12} e^{3\pi} (1 + e^{\pi})^2 \left(-1 + 2 e^{\pi}\right)  
        + \mathcal{O}_2(s).
\end{align*}

Note that $W_1 = W_2 = W_3 = 0$ and $W_4 \neq 0$ when $s_1 = s_2 = s_3 = 0$. Furthermore the Jacobian determinant of the first three Lyapunov quantities is nonvanishing:
\[
\det(J_s(W_1, W_2, W_3))(0) = \frac{\pi}{2} e^{3\pi} (1 + e^{\pi}) (e^{2\pi} - 1) \neq 0.
\]
By Proposition~\ref{aux_lemma} there exist sufficiently small real parameters $s_1, s_2, s_3$ such that system~\eqref{pw_fam11_3} has 3 limit cycles. Since $\phi^{-1}(0) = i$, system~\eqref{pw_fam1_3} consequently has 3 limit cycles bifurcating from the point $i$.
\end{proof}

Observe that in the proofs of the previous propositions, $G^\pm(0) = 0$, as in system \eqref{pwhs_eq_1}, and therefore sliding segments were not considered. In what follows by introducing a new small real parameter $d$ such that $\phi_{*}F^-(0) \neq 0$, a sliding segment appears. This mechanism gives rise to an additional limit cycle bifurcating from the origin, and consequently system~\eqref{ch4:eq111} will exhibit one more limit cycle bifurcating from the point $\phi^{-1}(0)=i$.

To prove this result we adapt to the piecewise holomorphic setting the ideas from~\cite{doi:10.1137/11083928X}, where the authors study the first return map near the origin for piecewise linear differential equations. In \cite{GASRONSIL1} Gasull et al. proved the following result for piecewise holomorphic systems.

\begin{proposition}[{\cite[Prop.~12]{GASRONSIL1}}]\label{prop_sliding_GAS}
Consider the following family of PWHS
\begin{equation}\label{pw_fam4_GAS}
\dot{w} =
\begin{cases}
(i+\lambda^+)w + \displaystyle\sum_{k=2}^\infty A_k^+ w^k, & \operatorname{Im}(w)\geq0,\\[6pt]
(i+\lambda^+)d + (i+\lambda^-)w + \displaystyle\sum_{k=2}^\infty A_k^- w^k, & \operatorname{Im}(w)\leq0,
\end{cases}
\end{equation}
where $d\in\mathbb{R}$ is a small parameter. When $\lambda^++\lambda^-$ is positive (resp. negative) a sufficiently small $d$ with $d<0$ (resp. $d>0$) yields a limit cycle bifurcating from the origin.
\end{proposition}

We now apply this result to each of the families studied in this section. The idea is to start from the $w$-system obtained via the pushforward of the Möbius map \eqref{mobius_map1}, then add the constant term $(i+\lambda^+)d$ in the lower half-plane as in \eqref{pw_fam4_GAS}, and finally return to the $z$-variable using $\phi^{-1}(w)=\frac{w+i}{iw+1}$. The resulting $z$-systems are presented below.

\medskip
\noindent\textbf{Linear--linear case (Proposition \ref{imp_prop31}).}
Applying Proposition~\ref{prop_sliding_GAS} to the $w$-system
\[
\dot{w} =
\begin{cases}
(i+\lambda+s)w + \widetilde{A}^+ w^2, & \operatorname{Im}(w)\geq0,\\[2mm]
(i+\lambda+s)d + (i-\lambda)w + \widetilde{A}^- w^2, & \operatorname{Im}(w)\leq0,
\end{cases}
\]
with $\widetilde{A}^+=-1+i(\lambda+s)$, $\widetilde{A}^-=-1-i\lambda+(\lambda+s+i)d$, and then returning to $z$ we obtain the following linear--linear system in $z$:
\[
\dot{z} =
\begin{cases}
(1-i(\lambda+s)) + (\lambda+s+i)z, & |z|\geq1,\\[2mm]
(1+i\lambda) + \big(i-\lambda + (2-2i(\lambda+s))d\big)z, & |z|\leq1.
\end{cases}
\]
For $d=0$ we recover the original system \eqref{pw_fam1}. Since $\lambda^++\lambda^- = s$ Proposition~\ref{prop_sliding_GAS} guarantees the bifurcation of one additional limit cycle from $i$ for sufficiently small $d$ with sign opposite to $\operatorname{sgn}(s)$. Thus the perturbed system has 2 limit cycles.

\medskip
\noindent\textbf{Quadratic--linear case (Proposition \ref{imp_prop32}).}
Starting from the $w$-system
\[
\dot{w} =
\begin{cases}
(2+s_1+i)w + \big(s_2-\tfrac{10}{3}i\big)w^2, & \operatorname{Im}(w)\geq0,\\[2mm]
(2+s_1+i)d + (-2+i)w + (-1-2i+(2+s_1+i)d)w^2, & \operatorname{Im}(w)\leq0,
\end{cases}
\]
and after applying $\phi^{-1}$ we obtain the perturbed quadratic--linear system in $z$:
\[
\dot{z} =
\begin{cases}
\frac{1}{2} - \frac{1}{2} s_2 + i\frac{(4 - 3 s_1)}{6}  -  \left( \frac{20}{6} + i s_2 \right) z +  \left( \frac{1}{2} + \frac{1}{2} s_2 - i\frac{(16 + 3 s_1)}{6} \right) z^2, & |z|\geq1,\\[4mm]
1+2i + \big(-2+i + d(2-i(4+2s_1))\big)z, & |z|\leq1.
\end{cases}
\]
The original system \eqref{pw_fam1_2} is obtained when $d=0$. Here $\lambda^++\lambda^- = s_1$, so for $d$ with sign opposite to $\operatorname{sgn}(s_1)$ and small enough the system has 3 limit cycles.

\medskip
\noindent\textbf{Quadratic--quadratic case (Proposition \ref{imp_prop3}).}
Applying the same procedure to the $w$-system
\[
\dot{w} = 
\begin{cases}
(1+s_1+i)w + (s_2+i)w^2, & \operatorname{Im}(w)\geq0,\\[2mm]
(1+s_1+i)d + (-1+i)w + i(1+s_3)w^2, & \operatorname{Im}(w)\leq0,
\end{cases}
\]
and applying $\phi^{-1}$ gives the following quadratic--quadratic system in $z$:
\[
\dot{z} = 
\begin{cases}
\dfrac{1}{2}\big((1-s_2)-i(s_1+2)\big) + (1-is_2)z + \dfrac{1}{2}\big(1+s_2-is_1\big)z^2, & |z|\geq1,\\[8pt]
\dfrac{1}{2}\big(1 - i s_3 + d(1+s_1+i)\big) + \big(1+s_3 - i d(1+s_1+i)\big)z & \\[2mm]
\qquad + \dfrac{1}{2}\big(1 + i(2+s_3) - d(1+s_1+i)\big)z^2, & |z|\leq1.
\end{cases}
\]
Setting $d=0$ gives back the unperturbed system \eqref{pw_fam1_3}. Here $\lambda^++\lambda^- = s_1$. Proposition \ref{prop_sliding_GAS} implies that for sufficiently small $d$ with sign opposite to $\operatorname{sgn}(s_1)$, one additional limit cycle bifurcates from $i$. Therefore the system exhibits 4 limit cycles.

\begin{proof}[Proof of Theorem \ref{teonew}]
     The three propositions above provide systems with $1$, $2$ and $3$ limit cycles respectively, by choosing the parameters  $s_1,s_2,$ and $s_3$ sufficiently small. Moreover, the sliding construction applied to each of these families adds one extra limit cycle, yielding systems with $2$, $3$ and $4$ limit cycles. Since all these limit cycles are hyperbolic (they arise from nondegenerate bifurcations), we obtain the lower bounds
\[
\mathcal{L}^0_{1,1}\ge 2,\qquad \mathcal{L}^0_{1,2}=\mathcal{L}^0_{2,1}\ge 3,\qquad \mathcal{L}^0_{2,2}\ge 4.
\]
This completes the proof of Theorem \ref{teonew}.
\end{proof}

\subsection{Proof of Theorem \ref{teo11}}
 First, we prove statement $(a)$. 
 Suppose that $f^\pm(z)=A^\pm z+B^\pm,$ where $A^\pm=a_1^\pm+ia_2^\pm,B^\pm=b_1^\pm+ib_2^\pm\in\mathbb{C}.$ Since $z=x+iy$, then the vector fields $F^\pm=\overline{f^\pm}$ can be written by $F^\pm(x,y)=u^\pm(x,y)-iv^\pm(x,y),$ with
 $$u^\pm(x,y)=b_1^\pm+a_1^\pm x-a_2^\pm y, v^\pm(x,y)=b_2^\pm+a_2^\pm x+a_1^\pm y.$$ We claim  that 
\begin{equation}\label{eq}
    \dot{x} = u^\pm(x,y), \quad
\dot{y} = -v^\pm(x,y),
\end{equation}
is a Hamiltonian system. Indeed, considering the function \begin{equation}\label{hamil1}
H^\pm(x,y)=b_2^\pm x+b_1^\pm y+\frac{a_2^\pm}{2}x^2+a_1^\pm xy-\frac{a_2^\pm}{2}y^2,    
\end{equation}
we get 
$$\frac{\partial H^\pm}{\partial y}(x,y)=u^\pm(x,y), \quad \frac{\partial H^\pm}{\partial x}(x,y)=v^\pm(x,y).$$
Therefore \eqref{eq} is a Hamiltonian system. 

Taking $A=\left(\frac{2x}{x^2+1},\frac{1-x^2}{x^2+1}\right),
\, B=\left(\frac{2y}{y^2+1},\frac{1-y^2}{y^2+1}\right)\in\mathbb{S}^1$ we define
\[
c^\pm(x,y)=\frac{H^\pm(A)-H^\pm(B)}{x-y},
\]
for \(x \neq y\). Recall that in the \((x,y)\)-plane, \(c^\pm(x,y)=0\) defines implicitly the half return map in \(\Sigma^\pm\). From the definition of $H^\pm$ given \eqref{hamil1}, we obtain
\[
c^\pm(x,y)=-\frac{2 \alpha^\pm(x, y)}{(1 + x^2)^2 (1 + y^2)^2},
\]
where 
\[
\begin{array}{rcl}
\alpha^\pm(x, y) & = &
-a_1^\pm - b_2^\pm - 2 a_2^\pm x + b_1^\pm x + a_1^\pm x^2 - b_2^\pm x^2 + b_1^\pm x^3 - 2 a_2^\pm y + b_1^\pm y \\
&& + 3 a_1^\pm x y + b_2^\pm x y + b_1^\pm x^2 y + a_1^\pm x^3 y + b_2^\pm x^3 y + a_1^\pm y^2 - b_2^\pm y^2 \\
&& + b_1^\pm x y^2 + 3 a_1^\pm x^2 y^2 - b_2^\pm x^2 y^2 + 2 a_2^\pm x^3 y^2 + b_1^\pm x^3 y^2 + b_1^\pm y^3 + a_1^\pm x y^3 \\
&& + b_2^\pm x y^3 + 2 a_2^\pm x^2 y^3 + b_1^\pm x^2 y^3 - a_1^\pm x^3 y^3 + b_2^\pm x^3 y^3.\\
\end{array}
\]
Hence \(c^\pm(x,y)=0\) implies that \(\alpha^\pm(x, y)=0\). In addition, the polynomial \(\alpha(x, y)\) is symmetric. Thus \((x_0, y_0)\) is a solution of \(\alpha^\pm(x, y)=0\) if, and only if \((y_0, x_0)\) also represents a solution. Consequently both solutions lead to the same possible limit cycle.

Furthermore, we know that if \( (x_0, y_0) \) satisfies \( \alpha^+(x, y)=\alpha^-(x, y) = 0 \), then \( x_0 \) must be a root of the resultant polynomial \( R(x) \) of \( \alpha^- \) and  \( \alpha^+ \) with respect to the variable \( y \). Similarly \( y_0 \) must be a root of the resultant polynomial \( S(y) \) of \( \alpha^- \) and  \( \alpha^+ \) concerning \( x \). Moreover, it is easy to see that $R(x)=S(x)$. Thus
\[
R(x) = -4 (1+x^2)^6(\beta_0+\beta_1x+\beta_2x^2+\beta_3x^3+\beta_4x^4+\beta_5x^5+\beta_6x^6),
\]
for certain coefficients $\beta_i=\beta_i(a_1^\pm,a_2^\pm,b_1^\pm,b_2^\pm)\in\mathbb{R},$ for $i=1,\dots,6$, which we omit here due to their size. Since the solutions of \((1+x^2)^6\) are complex we conclude that \(R(x)\) has at most 6 real roots. Due to the symmetric property we can conclude that the PWCS \eqref{ch4:eq111} can have at most three limit cycles.

Now we prove statement $(b)$. Assume that \( f^+(z) = A^\pm z^2 + B^\pm z+C^\pm \), where \( A^\pm = a_1^\pm + ia_2^\pm \), \( B^\pm = b_1^\pm + ib_2^\pm\),  \(C^\pm = c_1^\pm + ic_2^\pm \in \mathbb{C} \). Given that \( z = x + iy \) the vector fields \( F^\pm = \overline{f^\pm} \) can be expressed as \( F^\pm(x,y) = u^\pm(x,y) - iv^\pm(x,y) \), where
\[\begin{array}{rcl}
u^\pm(x,y) &=& c_1^\pm + b_1^\pm x + a_1^\pm x^2 - b_2^\pm y - 2 a_2^\pm x y - a_1^\pm y^2,\\
v^\pm(x,y) &=& c_2^\pm + b_2^\pm x + a_2^\pm x^2 + b_1^\pm y + 2 a_1^\pm x y - a_2^\pm y^2.
\end{array}\]
We assert that 
\begin{equation}\label{eq1}
\dot{x} = u^\pm(x,y), \quad \dot{y} = -v^\pm(x,y)
\end{equation}
describes a Hamiltonian system. Indeed, considering the function 

\begin{equation}\label{hamil2}
H^\pm(x,y) = c_2^\pm x + \frac{b_2^\pm x^2}{2} + \frac{a_2^\pm x^3}{3} + c_1^\pm y + b_1^\pm x y + a_1^\pm x^2 y - \frac{b_2^\pm y^2}{2} - a_2^\pm x y^2 - \frac{a_1^\pm y^3}{3}.
\end{equation}
we find that 

\[
\frac{\partial H^\pm}{\partial y}(x,y) = u^\pm(x,y), \quad \frac{\partial H^\pm}{\partial x}(x,y) = v^\pm(x,y).
\]
Thus \eqref{eq1} indeed represents a Hamiltonian system. From the definition of \( H^\pm \) provided in \eqref{hamil2} we get 
\[
c^\pm(x,y) = \frac{2 \alpha^\pm(x, y)}{3(1 + x^2)^3 (1 + y^2)^3},
\]
where 
\[
\begin{array}{rcl}
\alpha^\pm(x, y) &= &-3 a_2^- + 3 b_1^- + 3 c_2^- + 9 a_1^- x + 6 b_2^- x - 3 c_1^- x + 10 a_2^- x^2 + 6 c_2^- x^2 - 6 a_1^- x^3 \\
&&+ 6 b_2^- x^3 - 6 c_1^- x^3 - 3 a_2^- x^4 - 3 b_1^- x^4 + 3 c_2^- x^4 + a_1^- x^5 - 3 c_1^- x^5 + 9 a_1^- y \\
&&+ 6 b_2^- y - 3 c_1^- y + 19 a_2^- x y - 9 b_1^- x y - 3 c_2^- x y - 6 a_1^- x^2 y + 6 b_2^- x^2 y - 6 c_1^- x^2 y\\
&& + 6 a_2^- x^3 y - 12 b_1^- x^3 y - 6 c_2^- x^3 y + a_1^- x^4 y - 3 c_1^- x^4 y + 3 a_2^- x^5 y- 3 b_1^- x^5 y \\
&& - 3 c_2^- x^5 y + 10 a_2^- y^2 + 6 c_2^- y^2 - 6 a_1^- x y^2 + 6 b_2^- x y^2- 6 c_1^- x y^2 + 36 a_2^- x^2 y^2 \\
&& - 12 b_1^- x^2 y^2 + 12 c_2^- x^2 y^2 - 44 a_1^- x^3 y^2 - 12 c_1^- x^3 y^2- 6 a_2^- x^4 y^2 - 12 b_1^- x^4 y^2 \\
&& + 6 c_2^- x^4 y^2 - 6 a_1^- x^5 y^2 - 6 b_2^- x^5 y^2 - 6 c_1^- x^5 y^2- 6 a_1^- y^3 + 6 b_2^- y^3- 6 c_1^- y^3  \\
&& + 6 a_2^- x y^3 - 12 b_1^- x y^3 - 6 c_2^- x y^3- 44 a_1^- x^2 y^3 - 12 c_1^- x^2 y^3 - 36 a_2^- x^3 y^3  \\
&&- 12 b_1^- x^3 y^3 - 12 c_2^- x^3 y^3- 6 a_1^- x^4 y^3 - 6 b_2^- x^4 y^3 - 6 c_1^- x^4 y^3 - 10 a_2^- x^5 y^3\\
&& - 6 c_2^- x^5 y^3 - 3 a_2^- y^4 - 3 b_1^- y^4 + 3 c_2^- y^4 + a_1^- x y^4 - 3 c_1^- x y^4 - 6 a_2^- x^2 y^4 \\
&&- 12 b_1^- x^2 y^4 + 6 c_2^- x^2 y^4 - 6 a_1^- x^3 y^4 - 6 b_2^- x^3 y^4 - 6 c_1^- x^3 y^4- 19 a_2^- x^4 y^4 \\
&& - 9 b_1^- x^4 y^4 + 3 c_2^- x^4 y^4 + 9 a_1^- x^5 y^4 - 6 b_2^- x^5 y^4- 3 c_1^- x^5 y^4 + a_1^- y^5  - 3 c_1^- y^5 \\
&& + 3 a_2^- x y^5 - 3 b_1^- x y^5 - 3 c_2^- x y^5- 6 a_1^- x^2 y^5 - 6 b_2^- x^2 y^5 - 6 c_1^- x^2 y^5 - 10 a_2^- x^3 y^5  \\
&&- 6 c_2^- x^3 y^5 + 9 a_1^- x^4 y^5 - 6 b_2^- x^4 y^5 - 3 c_1^- x^4 y^5 + 3 a_2^- x^5 y^5 + 3 b_1^- x^5 y^5 - 3 c_2^- x^5 y^5.
\end{array}
\]
Therefore \( c^\pm(x,y) = 0 \) implies that \( \alpha^\pm(x, y) = 0 \). Additionally, the polynomial \( \alpha(x, y) \) is symmetric, so \( (x_0, y_0) \) is a solution of \( \alpha^\pm(x, y) = 0 \) if, and only if, \( (y_0, x_0) \) is also a solution. Consequently, both solutions lead to the same potential limit cycle.

Using the same ideas as in statement $(a)$ we know that if \( (x_0, y_0) \) satisfies \( \alpha^+(x, y) = \alpha^-(x, y) = 0 \), then \( x_0 \) must be a root of the resultant polynomial \( R(x) \) of \( \alpha^- \) and \( \alpha^+ \) with respect to the variable \( y \). Similarly \( y_0 \) must be a root of the resultant polynomial \( S(y) \) of \( \alpha^- \) and \( \alpha^+ \) concerning \( x \). Moreover, it is straightforward to see that \( R(x) = S(x) \). Thus 

\[
R(x) = 36864 (1+x^2)^{15} \sum_{j=1}^{20}\beta_j x^j,
\]
for certain coefficients \( \beta_j = \beta_j(a_1^\pm, a_2^\pm, b_1^\pm, b_2^\pm,c_1^\pm, c_2^\pm) \in \mathbb{R} \) with \( j = 1, \dots, 20 \), which we omit here due to their size. Since the solutions of \( (1+x^2)^{20} \) are complex we conclude that \( R(x) \) can have at most 20 real roots. Due to the symmetry property we can conclude that the PWCS \eqref{ch4:eq111} can have at most ten limit cycles.

\subsection{Proof of Theorem \ref{teod}}

The proof relies on the normal form classification given in 
Proposition~\ref{GGJ}. Each holomorphic vector field $F^\pm$ is locally
conformally conjugate to one of a finite collection of normal forms.

For each of these normal forms we establish, in the propositions below,
that the corresponding piecewise system cannot possess crossing limit
cycles. Since conformal conjugacies preserve trajectories and their
intersection properties with the switching circle, the absence of
crossing limit cycles for the normal forms implies the same property
for the original system.

\begin{proposition}\label{prop:linear}
Consider the piecewise complex system \eqref{ch4:eq111}. If one side of the 
switching circle is governed by the linear vector field 
$F(z)=(\alpha+i\beta)z$ with $\alpha,\beta\in\mathbb{R}$, then the system 
admits no crossing limit cycles.
\end{proposition}

\begin{proof}
We prove the case where the linear field acts in the exterior region $|z|\geq1$, 
the case where it acts in the interior $|z|\leq1$ is completely analogous.

Write $z=re^{i\theta}$. In the exterior region we have
\[
\dot{r}=\alpha r,\qquad \dot{\theta}=\beta.
\]
This system possesses the first integral $J(r, \theta) = \alpha\theta - \beta\ln r$, as verified by the computation:
\[
\frac{dJ}{dt} = \frac{\partial J}{\partial r}\dot{r} + \frac{\partial J}{\partial \theta}\dot{\theta} 
= \left(-\frac{\beta}{r}\right)(\alpha r) + (\alpha)(\beta) = 0.
\]

Assume, for contradiction, that a transversal crossing limit cycle exists. 
It must intersect the unit circle $\mathbb{S}^1$ at two distinct points 
$P_1=(1,\theta_1)$ and $P_2=(1,\theta_2)$ with $\theta_1\neq\theta_2$, and 
the exterior arc of the cycle connects these two points.

We distinguish two cases.

\paragraph{Case 1: $\alpha\neq0$.}  
Since $J$ is constant on the exterior arc, we have $J(1,\theta_1)=J(1,\theta_2)$, 
i.e. $\alpha\theta_1=\alpha\theta_2$. Because $\alpha\neq0$ this forces 
$\theta_1=\theta_2$, contradicting the hypothesis that the two intersection 
points are distinct.

\paragraph{Case 2: $\alpha=0$.}  
Then the exterior dynamics reduces to $\dot{r}=0$, $\dot{\theta}=\beta$. Every 
trajectory starting on $\mathbb{S}^1$ remains on the circle $r=1$ for all time. 
Consequently, an exterior arc starting at $P_1$ can never reach a different 
point $P_2$ of $\mathbb{S}^1$.

Both cases lead to a contradiction. Hence no crossing limit cycle 
can exist, regardless of the dynamics on the other side.

If the linear field acts in the interior $|z|\leq1$, the same argument applies 
to the interior arc, using the same first integral $J$ (which is also constant 
in $|z|\leq1$). The analysis is identical.
\end{proof}

\begin{proposition}\label{prop:monomial}
Consider the piecewise complex system \eqref{ch4:eq111}. Suppose $F^-(z)=z^n$ with $n\in \mathbb{Z}\setminus\{1\}$ and $F^+(z)=z^m$ with $m\in \mathbb{Z}\setminus\{1\}$.
Assume that:
\begin{enumerate}
\item[(i)] There exists an orbit of $\dot z = F^-(z)$ contained in $|z|\leq1$ connecting two distinct points $p,q\in \mathbb{S}^1$.
\item[(ii)] There exists an orbit of $\dot z = F^+(z)$ contained in $|z|\geq1$ connecting the same points $p,q\in \mathbb{S}^1$.
\end{enumerate}
Then:
\begin{enumerate}
\item[(a)] The points $p,q$ are symmetric with respect to some line through the origin.
\item[(b)] Each of the fields $F^-$ and $F^+$ possesses infinitely many orbits connecting $\mathbb{S}^1$ to itself in its respective region.
\item[(c)] The piecewise system has a \emph{center}: a continuous family of periodic orbits covering a neighbourhood of $\mathbb{S}^1$.
\item[(d)] In particular, no limit cycle exists.
\end{enumerate}
\end{proposition}

\begin{proof}
Write $z=re^{i\theta}$. For an integer $k\neq 1$ the field $z'=z^k$ in polar coordinates satisfies
\[
r' = r^k\cos((k-1)\theta), \qquad \theta' = r^{k-1}\sin((k-1)\theta),
\]
and a direct computation shows that
$I_k(r,\theta) = r^{1-k}\sin((1-k)\theta)$
is a first integral, constant along every orbit.

For the interior field $F^-(z)=z^n$ and the exterior field $F^+(z)=z^m$ with $n,m\in\mathbb Z\setminus\{1\}$, set
\[
\beta = 1-n,\quad \alpha = 1-m,\quad I^-(\theta) = \sin(\beta\theta),\quad I^+(\theta) = \sin(\alpha\theta),
\]
the restrictions of the first integrals to the unit circle $\mathbb{S}^1$.

Let $p=e^{i\theta_1}$ and $q=e^{i\theta_2}$ be two distinct points of $\mathbb{S}^1$ connected by an orbit of $F^-$ contained in $|z|\leq1$. Since $I^-$ is constant on this orbit we have $I^-(\theta_1)=I^-(\theta_2)$, i.e.\ $\sin(\beta\theta_1)=\sin(\beta\theta_2)$. Hence there exists an integer $k$ such that $\beta\theta_2 = \pi - \beta\theta_1 + 2k\pi$, which yields $\theta_2 = c - \theta_1$ with $c = (\pi+2k\pi)/\beta$. Thus $\theta_2$ is the image of $\theta_1$ under the reflection $\sigma(\theta)=c-\theta$ of the circle. The same argument applied to the exterior orbit of $F^+$ connecting $p$ and $q$ shows that the same reflection $\sigma$ sends $\theta_1$ to $\theta_2$. Therefore $p$ and $q$ are symmetric with respect to a fixed line through the origin determined by $\sigma$, proving statement (a).

Now consider the interior field $F^-$. For each value $c$ in the range of $I^-$ on $\mathbb{S}^1$, the equation $I^-(\theta)=c$ has exactly two solutions, $\theta$ and $\sigma(\theta)$. The corresponding level set $I^-(r,\theta)=c$ is an orbit of $F^-$, because $r^{n-1} = \sin(\beta\theta)/c$ and for all $\theta$ on this curve we have $|\sin(\beta\theta)| \le |c|$, it follows that $r \le 1$ on the whole orbit, so the orbit lies entirely in $|z|\leq1$. Hence each such $c$ gives an interior orbit connecting the two symmetric points of $\mathbb{S}^1$. Since the range of $I^-$ contains an interval there are infinitely many such values $c$ and therefore infinitely many interior connections. The same reasoning applied to $F^+$ shows that there are also infinitely many exterior orbits in $|z|\geq1$ connecting symmetric points of $\mathbb{S}^1$. This establishes statement (b).

Let $p = e^{i\theta_1}$ and $q = e^{i\sigma(\theta_1)}$ be the points connected by the given interior and exterior orbits, and set $c_0 = I^+(\theta_1)$. Since the exterior orbit is contained in $|z|\geq1$, the corresponding level set of $I^+$ lies entirely in $|z|\ge 1$. If $|c_0|=1$ then $\theta_1$ is a maximum or minimum of $I^+$; for $\theta$ sufficiently close to $\theta_1$ but different from it we have $|I^+(\theta)|<1$, and the corresponding exterior orbit stays strictly in $|z|\geq1$. For such $\theta$ define $\tilde p = e^{i\theta}$ and $\tilde q = e^{i\sigma(\theta)}$. By statement (b) there exists an interior orbit of $F^-$ connecting $\tilde p$ to $\tilde q$ and an exterior orbit of $F^+$ connecting $\tilde q$ back to $\tilde p$. Concatenating these two arcs yields a periodic orbit of the piecewise system. Varying $\theta$ continuously produces a continuous family of periodic orbits that fills a neighbourhood of $\mathbb{S}^1$; thus the system possesses a center, proving statement (c).

Finally, any periodic orbit that crosses $\mathbb{S}^1$ must intersect the circle at two points. By the interior dynamics these two points must be symmetric with respect to $\sigma$. The exterior arc joining them is uniquely determined by the level set of $I^+$, and therefore every periodic orbit belongs to the continuous family constructed above. Consequently no periodic orbit is isolated, and the system admits no limit cycles. This completes the proof of statement (d).
\end{proof}

\begin{proposition}
    \label{prop:rational}
Let \(n\in \mathbb{N}\), \(n\ge 2\), and \(\gamma\in\mathbb{R}\) with \(|\gamma|\ge 1\). 
Consider the holomorphic vector field
\[
F_\gamma(z)=\frac{z^n}{1+\gamma z^{\,n-1}},
\]
defined on \(\Omega=\mathbb{C}\setminus\{z:1+\gamma z^{n-1}=0\}\).  
Denote by 
\(\mathbb{S}^1_*=\mathbb{S}^1\setminus\{z:1+\gamma z^{n-1}=0\}\) its regular part.  
Assume that \(F_\gamma\) is assigned to one side of the unit circle, either to 
the interior region \(|z|\leq1\) or to the exterior region \(|z|\geq1\), while the other 
side is filled with an arbitrary holomorphic dynamics.

Then the piecewise system admits no transversal crossing limit cycle
intersecting \(\mathbb{S}^1_*\) in two distinct points.
\end{proposition}

\begin{proof}
We work in polar coordinates \(z=re^{i\theta}\). For any \(z\in\Omega\) we have
\[
\dot z = F_\gamma(z)=\frac{r^n e^{in\theta}}{1+\gamma r^{n-1}e^{i(n-1)\theta}}.
\]

The radial component is
$\dot r=\operatorname{Re}(e^{-i\theta}\dot z).$
A direct computation gives
\[
\dot r = \operatorname{Re}\!\left(
\frac{r^n e^{i(n-1)\theta}}
{1+\gamma r^{n-1}e^{i(n-1)\theta}}
\right).
\]

On the unit circle \(r=1\) and away from the poles (i.e. on \(\mathbb{S}^1_*\)), set
\(\varphi=(n-1)\theta\). Then
\[
\dot r\big|_{r=1} =
\operatorname{Re}\!\left(
\frac{e^{i\varphi}}{1+\gamma e^{i\varphi}}
\right).
\]

Multiplying numerator and denominator by the complex conjugate of 
\(1+\gamma e^{i\varphi}\) we obtain
\[
\frac{e^{i\varphi}}{1+\gamma e^{i\varphi}}
=
\frac{e^{i\varphi}(1+\gamma e^{-i\varphi})}
{|1+\gamma e^{i\varphi}|^2}
=
\frac{e^{i\varphi}+\gamma}{|1+\gamma e^{i\varphi}|^2}.
\]

Taking the real part yields
\[
\dot r\big|_{\mathbb{S}^1_*}
=
\frac{\cos\varphi+\gamma}{|1+\gamma e^{i\varphi}|^2}.
\]

Since the denominator is positive on \(\mathbb{S}^1_*\), the sign of \(\dot r\) is the 
sign of \(\cos\varphi+\gamma\).

If \(\gamma\ge 1\), then
$\cos\varphi+\gamma \ge \gamma-1 \ge 0,$
and the equality occurs only when \(\gamma=1\) and \(\cos\varphi=-1\), which 
correspond exactly to the poles (excluded from \(\mathbb{S}^1_*\)). Hence
$\dot r>0$ on $\mathbb{S}^1_*.$

If \(\gamma\le -1\), then
$\cos\varphi+\gamma \le \gamma+1 \le 0,$
with the equality only at the poles. Hence
$\dot r<0$ on $\mathbb{S}^1_*.$

Therefore whenever \(|\gamma|\ge1\) the radial velocity has a constant sign 
on the regular part of the unit circle.

\medskip

\noindent
\textbf{Claim.}
Every trajectory of \(F_\gamma\) intersects \(\mathbb{S}^1_*\) at most once.

Indeed, if \(\gamma\ge1\) any trajectory that reaches \(\mathbb{S}^1_*\) from inside 
(\(r<1\)) satisfies \(\dot r>0\) at the intersection and therefore immediately 
exits to \(r>1\). It cannot return to the circle, because returning would require 
approaching \(\mathbb{S}^1\) from outside with \(\dot r\le0\), contradicting the positivity 
of \(\dot r\) at every point of \(\mathbb{S}^1_*\). Moreover no trajectory can arrive 
from outside because approaching the circle from \(r>1\) would require 
\(\dot r\le0\) just before the contact.

If \(\gamma\le-1\) the situation is symmetric: any trajectory meeting 
\(\mathbb{S}^1_*\) from outside immediately enters the disk and cannot return.

Consequently no trajectory of \(F_\gamma\) contained entirely in one side of 
the circle can connect two distinct points of \(\mathbb{S}^1_*\).

\medskip

Now suppose, for contradiction, that a transversal crossing limit cycle exists. 
Such a cycle must intersect \(\mathbb{S}^1\) at two distinct regular points, one of them 
is an \textit{exit} and the other an \textit{entry}. We consider the two possible 
placements of \(F_\gamma\).

\paragraph{Case 1: \(F_\gamma\) governs the interior (\(|z|\leq1\)).}
The interior segment of the periodic orbit must be a trajectory of \(F_\gamma\) 
contained in \(|z|\leq1\) connecting the entry point to the exit point. This would 
be a trajectory of \(F_\gamma\) joining two distinct points of \(\mathbb{S}^1_*\), which 
is impossible by the property proved above.

\paragraph{Case 2: \(F_\gamma\) governs the exterior (\(|z|\geq1\)).}
In this case the exterior segment of the cycle would be a trajectory of 
\(F_\gamma\) contained in \(|z|\geq1\) connecting the entry point to the exit point. 
Again this would require a trajectory of \(F_\gamma\) intersecting \(\mathbb{S}^1_*\) in 
two distinct points, which is impossible.

Both cases lead to a contradiction. Therefore no transversal crossing limit 
cycle can exist.

Finally, the poles of \(F_\gamma\) on \(\mathbb{S}^1\) are isolated singularities and 
cannot serve as crossing points of a continuous periodic orbit.
\end{proof}

We are now in position to prove Theorem \ref{teod}.

\begin{proof}[Proof of Theorem \ref{teod}]
By Proposition~\ref{GGJ} each vector field $F^\pm$ is locally conformally
conjugate to one of the following normal forms:

\begin{enumerate}
\item $1$ (regular case);
\item $(\alpha+i\beta)z$, $\alpha,\beta\in\mathbb{R}$ (simple zero);
\item $z^n$, $n\ge2$ (zero of higher order with vanishing residue);
\item $\displaystyle \frac{z^n}{1+\gamma z^{\,n-1}}$, 
      where $\gamma=\operatorname{Res}(1/F^\pm,0)\neq0$;
\item $z^{-n}$, $n\ge1$ (pole of order $n$).
\end{enumerate}

The absence of crossing limit cycles for each of these normal forms
has been established in the propositions proved above, namely
Propositions~\ref{prop:linear}, \ref{prop:monomial} and
\ref{prop:rational}.

Since conformal conjugacies preserve trajectories and their intersection
properties with the unit circle, the existence of a crossing limit cycle
for the original system would imply the existence of such a cycle for
one of the normal forms, contradicting the results of
Propositions~\ref{prop:linear}--\ref{prop:rational}.

Therefore the system \eqref{ch4:eq111} admits no crossing limit cycles.
\end{proof}

\subsubsection{Proof of Theorem \ref{thm:alg_s1}}\label{sec:alg_cycles}

The proof is constructive. We first exhibit a piecewise holomorphic system with switching manifold $\mathbb{R}$ that admits an algebraic limit cycle, and then apply a M\"obius transformation to obtain the desired system with circular discontinuity.

Consider the system
\[
\dot{w} =
\begin{cases}
i w, & \operatorname{Im}(w) \geq 0, \\[6pt]
\left(3 + \dfrac{2i}{3}\right) + i w - \left(\dfrac{3}{4} - \dfrac{i}{6}\right) w^2, & \operatorname{Im}(w) \leq 0,
\end{cases}
\]
where the switching manifold is the real axis $\mathbb{R} = \{w\in\mathbb{C}: \operatorname{Im}(w) = 0 \}$.

In the upper half-plane the vector field is $\dot{w}=i w$. Setting $w = 2 e^{i\theta}$ we have $\dot{w}=2i e^{i\theta}$ and $\bar{w}=2 e^{-i\theta}$. Hence
\[
\frac{d}{dt}|w|^2 = \dot{w}\bar{w} + w\dot{\bar{w}} = (2i e^{i\theta})(2 e^{-i\theta}) + (2 e^{i\theta})(-2i e^{-i\theta}) = 4i - 4i = 0,
\]
so $|w|$ is constant and the circle $|w|=2$ is invariant in the upper half-plane. In the lower half-plane, setting $w = 2 e^{i\theta}$ we compute
\[
\dot{w}\bar{w} = \left[\left(3 + \frac{2i}{3}\right) + 2i e^{i\theta} - \left(\frac{3}{4} - \frac{i}{6}\right) 4 e^{2i\theta}\right] 2 e^{-i\theta} = 2\left(3 + \frac{2i}{3}\right)e^{-i\theta} + 4i - 8\left(\frac{3}{4} - \frac{i}{6}\right)e^{i\theta}.
\]
Taking the complex conjugate
\[
\overline{\dot{w}\bar{w}} = 2\left(3 - \frac{2i}{3}\right)e^{i\theta} - 4i - 8\left(\frac{3}{4} + \frac{i}{6}\right)e^{-i\theta}.
\]
Adding the two expressions yields $\frac{d}{dt}|w|^2 = \dot{w}\bar{w} + \overline{\dot{w}\bar{w}} = 0$, hence $|w|=2$ is also invariant in the lower half-plane.

In Cartesian coordinates $(x,y)$, with $w=x+iy$, the lower half-plane system becomes
\[
\dot{x} = 3 - y - \frac{3}{4}x^2 + \frac{3}{4}y^2 - \frac{1}{3}xy, \qquad
\dot{y} = \frac{2}{3} + x - \frac{3}{2}xy + \frac{1}{6}x^2 - \frac{1}{6}y^2.
\]
Define $S(x,y)=x^2+y^2-4$ and $L(x,y)= \frac{1}{3}x - \frac{3}{2}y + 1$. A direct computation yields
\[
\dot{S} = -2\left(\frac{3}{2}x + \frac{1}{3}y\right)S, \qquad
\dot{L} = -2\left(\frac{3}{2}x + \frac{1}{3}y\right)L.
\]
Hence the cofactors satisfy $\mathcal{K}_S = \mathcal{K}_L= -2\left(\frac{3}{2}x + \frac{1}{3}y\right)$, and consequently the function
\[
H(x,y) = \frac{S(x,y)}{L(x,y)} = \frac{x^2+y^2-4}{\frac{1}{3}x - \frac{3}{2}y + 1}
\]
is a first integral in the lower half-plane, because 
$$
\frac{dH}{dt} = \frac{\dot{S}L - S\dot{L}}{L^2} = \frac{(\mathcal{K}_SS)L - S(\mathcal{K}_LL)}{L^2} = \frac{(\mathcal{K}_SS)L - S(\mathcal{K}_SL)}{L^2} = 0.
$$
Take a transversal section $\Sigma_0$ on the positive real axis, i.e. $\Sigma_0 = \{ u \in \mathbb{R} \mid u > 0 \}$. For a point $u \in \Sigma_0$ the flow in the upper half-plane under $\dot{w}=i w$ follows the circle $|w|=u$ and returns to the real axis at $-u$. Hence the upper half-plane return map is $\pi^+(u) = -u$, which takes values on the negative real axis.

To compute the lower half-plane return map we use the first integral $H$. On the real axis ($y=0$) we have
$H(x,0) = (x^2-4)/(\frac{1}{3}x + 1).$
Since $H$ is constant along trajectories in the lower half-plane if a trajectory starts at a point $u$ on the real axis and returns to $v$ (also on the real axis) we have
\[
\frac{u^2-4}{\frac{1}{3}u+1} = \frac{v^2-4}{\frac{1}{3}v+1}.
\]
Solving for $v$ yields two solutions: $v = u$ (the trivial fixed point) and
$v = \pi^-(u) = \frac{-3u - 4}{u + 3},$
which is defined for all $u \neq -3$. In particular, for a point on the negative real axis $u = -u_0$ with $u_0 > 0$ we have
\[
\pi^-(-u_0) = \frac{3u_0 - 4}{-u_0 + 3}.
\]

The total Poincaré map is the composition $\pi = \pi^- \circ \pi^+$, i.e. for $u \in \Sigma_0$
\[
\pi(u) = \pi^-(-u) = \frac{3u - 4}{-u + 3}.
\]

The fixed points of $\pi$ satisfy $\pi(u)=u$. Solving
\[
\frac{3u - 4}{-u + 3} = u \quad \Longrightarrow \quad 3u - 4 = u(-u + 3) = -u^2 + 3u \quad \Longrightarrow \quad u^2 = 4,
\]
hence $u = 2$ (since $u>0$). Moreover
\[
\pi'(u) = \frac{5}{(3-u)^2}, \qquad \pi'(2) = 5 \neq 1.
\]
Thus the fixed point is hyperbolic and therefore isolated.

The resulting limit cycle $\Gamma$ is the circle $|w|=2$. Indeed, the orbit starting at $w=2$ follows the circle in the upper half-plane under $\dot{w}=i w$ until it reaches $w=-2$, then follows the circle in the lower half-plane under $\dot{w}=F^-(w)$ back to $w=2$. The circle $|w|=2$ is invariant in both half-planes, and the first integral $H$ guarantees that the lower half-plane trajectory returns exactly to the starting point. Thus $\Gamma = \{w\in\mathbb{C}: |w| = 2 \}$ is an algebraic limit cycle.

By the Fundamental Lemma~\ref{pushforward}, the pushforward of the system under $\phi^{-1}$ yields a piecewise holomorphic system with switching manifold $\mathbb{S}^1$, where $\phi$ is the Möbius transformation introduced in~\eqref{mobius_map1}, see Figure \ref{fig_cal12}.

For $|z|\geq1$ (the image of $\operatorname{Im}(w)\geq0$) the vector field becomes $\dot{z} = (z^2+1)/2$. For $|z|\leq1$ (the image of $\operatorname{Im}(w)\leq0$), a direct computation using the pushforward formula gives
\[
\dot{z} = \left(\frac{19}{8} + \frac{i}{4}\right) + \left(\frac{5}{6} - \frac{9i}{4}\right)z - \left(\frac{11}{8} + \frac{i}{4}\right)z^2.
\]

The limit cycle $\Gamma = \{w\in\mathbb{C}: |w| = 2 \}$ is mapped to
\[
\phi^{-1}(\Gamma) = \left\{ \frac{w+i}{iw+1} \mid |w| = 2 \right\},
\]
which is the circle with center $(0, -5/3)$ and radius $4/3$. In complex form this cycle is given by $3z\bar{z} - 5i(z - \bar{z}) + 3 = 0$, or equivalently $x^2 + (y+5/3)^2 = (4/3)^2$ in real coordinates. This is an algebraic curve (a circle) and, by Proposition \ref{prop:alg_mobius} it is a limit cycle of the transformed system. Thus we have constructed a PWHS separated by $\mathbb{S}^1$ with an algebraic limit cycle.

\begin{figure}[htbp]
\centering
\begin{overpic}[width=0.9\textwidth]{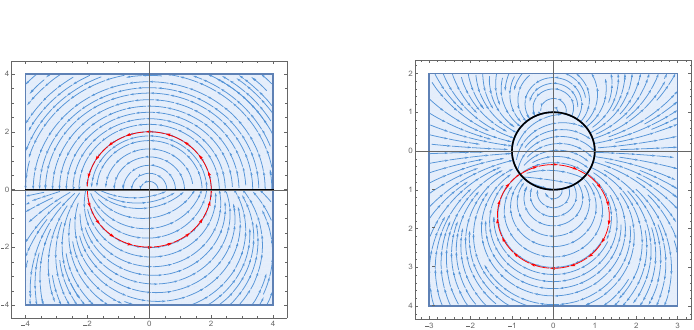}
    \put(15,40){\footnotesize $\phi(\Sigma) = \mathbb{R}$}
    
    \put(45,20){$\xrightarrow{\quad \phi^{-1} \quad}$}
    
    \put(76,40){$\Sigma = \mathbb{S}^1$}
\end{overpic}
\caption{The Möbius transformation $\phi^{-1}(w) = \dfrac{w+i}{iw+1}$ maps the system with switching manifold $\mathbb{R}$ (left) to the system with switching manifold $\mathbb{S}^1$ (right). The limit cycle $|w|=2$ (red) is mapped to the algebraic curve $3z\bar{z} - 5i(z-\bar{z}) + 3 = 0$, which in real coordinates is the circle $x^2 + (y+5/3)^2 = (4/3)^2$ (red).}
\label{fig_cal12}
\end{figure}

\section{Acknowledgements}

The authors would like to express their sincere gratitude to Armengol Gasull for suggesting the study of algebraic limit cycles in holomorphic systems and for his valuable insights that greatly contributed to this work.

Gabriel Rondón is also partially supported by the Agencia Estatal de Investigaci\'on of Spain grant PID2022-136613NB-100 and FAPESP grant 2024/15612-6.

Paulo R. da Silva is partially supported by ANR-23-CE40-0028, 
FAPESP grant 2023/02959-5,  FAPESP grant 2024/15612-6, CAPES/MATH-Amsud (88881.179491/2025-01) and CNPq grant 302154/2022-1.

Jaume Llibre is partially supported by to the Agencia Estatal de Investigación of Spain grant PID2022-136613NB-100.

\bibliographystyle{abbrv}
\bibliography{references1}

\end{document}